\documentclass[11pt]{article}
\normalsize
\usepackage{amsmath,amssymb,amsthm,mathtools}
\usepackage[hidelinks]{hyperref}
\hypersetup{
  pdftitle={Modified induction and a torsion-theoretic equivalence for relative BiHom-Hopf modules},
  pdfauthor={Qihao Jin},
  pdfkeywords={BiHom-Hopf algebra, relative BiHom-Hopf module, induction functor, coinvariants, torsion theory}
}

\theoremstyle{plain}
\newtheorem{theorem}{Theorem}[section]
\newtheorem{lemma}[theorem]{Lemma}
\newtheorem{proposition}[theorem]{Proposition}
\newtheorem{corollary}[theorem]{Corollary}

\theoremstyle{definition}
\newtheorem{definition}[theorem]{Definition}
\newtheorem{example}[theorem]{Example}
\newtheorem{remark}[theorem]{Remark}

\title{Modified induction and a torsion-theoretic equivalence for relative BiHom-Hopf modules}
\author{Qihao Jin\thanks{Haide College, Ocean University of China, Qingdao 266100, P. R. China. Email: \texttt{jinqihao@stu.ouc.edu.cn}}}
\date{}

\begin{document}
\maketitle

\begin{abstract}
We develop an induction theory for relative BiHom-Hopf modules, the special BiHom-Doi--Hopf case associated with the datum $(H,A,H)$ in which the final copy of $H$ carries its regular right $H$-module coalgebra structure. Let $H$ be a monoidal BiHom-Hopf algebra, let $A$ be a right $H$-BiHom-comodule algebra, and let $B=A^{\operatorname{co}H}$. We show that the balanced tensor product defines an induction functor $-\otimes_BA$ left adjoint to the coinvariant functor. If $H$ has a fixed Haar integral, untwisting yields a Haar identity for the BiHom setting and a canonical projection onto coinvariants. These constructions define a hereditary torsion theory with radical $\kappa$ and a torsion-free reflector $Q(M)=M/\kappa(M)$. The modified induction functor $Q(-\otimes_BA)$ then gives an equivalence between right $B$-BiHom-modules and torsion-free relative BiHom-Hopf modules generated by their coinvariants. Equal structure maps recover the corresponding Hom result, while identity structure maps recover the classical relative Hopf-module setting.
\end{abstract}

\noindent\textbf{Keywords.} BiHom-Hopf algebra; relative BiHom-Hopf module; induction functor; coinvariants; torsion theory; categorical equivalence.

\noindent\textbf{2020 Mathematics Subject Classification.} Primary 16T05; Secondary 18A40, 18E40.

\section{Introduction}

Hom-type structures twist classical algebraic identities by distinguished linear self-maps. Originating in deformation theory and monoidal category theory, Hom-algebras, coalgebras, bialgebras, and Hopf-type objects have since been studied extensively \cite{HartwigLarssonSilvestrov2006,MakhloufSilvestrov2008,MakhloufSilvestrov2010,CaenepeelGoyvaerts2011,Yau2009HomYBE,Yau2011HomYBE,Yau2012HomQuantumI}.

BiHom structures replace the single twisting map with two commuting maps. With bijective structure maps, they admit a monoidal interpretation as twists of ordinary algebraic structures. This viewpoint has motivated work on BiHom-associative algebras, bialgebras, Hopf algebras, Yang--Baxter operators, Yetter--Drinfeld categories, smash products, and related representation-theoretic constructions \cite{GrazianiMakhloufMeniniPanaite2015,FangLiu2018BiHomYBE,LiuMakhloufMeniniPanaite2020BiHomNovikov,LiuMakhloufMeniniPanaite2021BiHomPreLie,LiuShen2022BiHomYD,LuWangLiu2023BiHomLRSmash,ShenLiu2021BiHomSmash,WuZhang2022BiHomNonlinear,ZhangWangChen2024BiHomHopf,GaiLiSun2025,SunWangZhu2025}.

Doi--Hopf modules provide a common framework for modules and comodules over Hopf algebras and comodule algebras. In the classical setting, induction, coinvariants, and quotient constructions connect relative Hopf modules with Hopf--Galois theory and stable Clifford theory \cite{Sweedler1969,DoiTakeuchi1986,Doi1992,CaenepeelMilitaruZhu1997,Montgomery1993,DascalescuNastasescuRaianu2001,CaenepeelRaianuVanOystaeyen1994,VanOystaeyenZhang1996}. Hom analogs include monoidal Hom-Hopf algebras, Hom-Doi--Hopf modules, relative Hom-Hopf modules, induction functors, and duality constructions \cite{CaenepeelGoyvaerts2011,GuoZhangWang2014DoiHom,GuoZhangWang2015RelativeHom,JiaXu2022,SunWangZhang2025}. In the terminology of a general Doi datum $(H,A,C)$, the category studied below is the relative case $C=H$, with the regular right $H$-module coalgebra structure on $H$; we retain the notation $\mathcal M_A^H$ and occasionally the phrase ``BiHom-Doi--Hopf module'' to facilitate comparison with the Hom literature.

We extend this induction theory to the BiHom setting. Let $H$ be a monoidal BiHom-Hopf algebra, let $A$ be a right $H$-BiHom-comodule algebra, and put $B=A^{\operatorname{co}H}$. For every right $B$-BiHom-module $N$, we show that $N\otimes_BA$ is naturally a right relative BiHom-Hopf module and obtain the adjunction
\[
F=-\otimes_BA:\mathcal M_B\rightleftarrows\mathcal M_A^H:
G=(-)^{\operatorname{co}H},
\qquad F\dashv G.
\]

If $H$ admits a Haar integral, we construct a canonical projection from each $M\in\mathcal{M}_A^H$ onto its coinvariants and an associated subobject $\kappa(M)$. An untwisting argument transports the BiHom-Hopf structure to an ordinary Hopf algebra and yields the Haar identity needed in the BiHom setting. This identity allows us to prove that $M\mapsto\kappa(M)$ defines a hereditary torsion theory on $\mathcal{M}_A^H$ in the sense of Dickson \cite{Dickson1966}.

We then pass to the torsion-free factor and introduce the modified induction functor
\[
\overline{F}(N)=\overline{N\otimes_B A}=(N\otimes_B A)/\kappa(N\otimes_B A).
\]
Our main theorem shows that $\overline{F}$ and $G$ induce an equivalence between $\mathcal{M}_B$ and the full subcategory of torsion-free, $0$-generated objects in $\mathcal{M}_A^H$. Our argument parallels the Hom theory of Jia and Xu \cite{JiaXu2022}, but two independent structure maps must be tracked in every balancing and colinearity calculation. The key technical ingredient is the transported Haar identity in Lemma~\ref{lem:bihom-haar-identities}, which controls both twists in proving that $\kappa(M)$ is a subcomodule. We also make the descent and naturality arguments explicit and identify $Q$ as the torsion-free reflector. Equal structure maps recover the Hom case, while identity maps recover the classical relative Hopf-module case.

Section~2 fixes conventions, Section~3 develops induction and torsion theory, and Section~4 proves the equivalence.

\section{Preliminaries}

We work over a fixed field $\mathbb{k}$, and all tensor products are taken over $\mathbb{k}$. We use Sweedler notation $\Delta(c)=c_1\otimes c_2$ and $\rho_M(m)=m_{[0]}\otimes m_{[1]}$. Throughout, we work in the monoidal BiHom setting and assume that all structure maps are bijective. Following \cite{GrazianiMakhloufMeniniPanaite2015}, we denote the algebra twists by $\alpha,\beta$ and the coalgebra twists by $\omega=\alpha^{-1}$ and $\psi=\beta^{-1}$. In the tuples below, we suppress $\omega,\psi$ and express all coalgebra and comodule identities using $\alpha^{-1},\beta^{-1}$; this is only a notational convention.

\begin{definition}
A \textit{monoidal BiHom-associative algebra} is a tuple $(A,\mu_A,\alpha_A,\beta_A,1_A)$,
where $A$ is a $\mathbb{k}$-vector space, $\mu_A:A\otimes A\to A$ is a
linear map, $\mu_A(a\otimes b)=ab$, $\alpha_A,\beta_A\in
\operatorname{Aut}_{\mathbb{k}}(A)$ are commuting linear automorphisms, and
$1_A\in A$, such that, for all $a,b,c\in A$,
\begin{align}
&\alpha_A(ab)=\alpha_A(a)\alpha_A(b),
\qquad \beta_A(ab)=\beta_A(a)\beta_A(b), \label{eq:bihom-alg-mult}\\
&\alpha_A(a)(bc)=(ab)\beta_A(c), \label{eq:bihom-assoc}\\
&\alpha_A(1_A)=1_A,\qquad \qquad  \quad \,\,\,\,\,
\beta_A(1_A)=1_A, \label{eq:bihom-unit-stable}\\
&a1_A=\alpha_A(a), \qquad \qquad  \quad \,\,\,\,\,
1_Aa=\beta_A(a). \label{eq:bihom-unit}
\end{align}
\end{definition}

\begin{remark}
If $\alpha_A=\beta_A$, this definition specializes to a monoidal Hom-associative algebra. If $\alpha_A=\beta_A=\operatorname{id}_A$, it specializes to an ordinary unital associative algebra.
\end{remark}

\begin{definition}
A \textit{monoidal BiHom-coassociative coalgebra} is a tuple $(C,\Delta_C,\alpha_C,\beta_C,\varepsilon_C)$, where $C$ is a $\mathbb{k}$-vector space, $\Delta_C:C\to C\otimes C$ and $\varepsilon_C:C\to\mathbb{k}$ are linear maps, and $\alpha_C,\beta_C\in\operatorname{Aut}_{\mathbb{k}}(C)$ are commuting linear automorphisms, such that
\begin{align}
&\Delta_C\circ\alpha_C
=(\alpha_C\otimes\alpha_C)\circ\Delta_C, \qquad
\Delta_C\circ\beta_C
=(\beta_C\otimes\beta_C)\circ\Delta_C, \label{eq:bihom-coalg-comult}\\
&(\Delta_C\otimes\beta_C^{-1})\circ\Delta_C
=(\alpha_C^{-1}\otimes\Delta_C)\circ\Delta_C, \label{eq:bihom-coassoc}\\
&\varepsilon_C\circ\alpha_C
=\varepsilon_C, \qquad \qquad \qquad \qquad \,\,\,
\varepsilon_C\circ\beta_C
=\varepsilon_C, \label{eq:bihom-counit-inv}\\
&(\operatorname{id}_C\otimes\varepsilon_C)\circ\Delta_C
=\alpha_C^{-1},   \qquad \qquad \,\,
(\varepsilon_C\otimes\operatorname{id}_C)\circ\Delta_C
=\beta_C^{-1}. \label{eq:bihom-counit}
\end{align}
Equivalently, in Sweedler notation,
\[
\alpha_C^{-1}(c_1)\otimes c_{2\,1}\otimes c_{2\,2} = c_{1\,1}\otimes c_{1\,2}\otimes\beta_C^{-1}(c_2),
\]
and
\[
c_1\varepsilon_C(c_2)=\alpha_C^{-1}(c), \qquad \varepsilon_C(c_1)c_2=\beta_C^{-1}(c).
\]
\end{definition}

\begin{definition}
A \textit{monoidal BiHom-bialgebra} is a tuple $(H,\mu_H,\Delta_H,\alpha_H,\beta_H,1_H,\varepsilon_H)$ such that $(H,\mu_H,\alpha_H,\beta_H,1_H)$ is a monoidal BiHom-associative algebra, $(H,\Delta_H,\alpha_H,\beta_H,\varepsilon_H)$ is a monoidal BiHom-coassociative coalgebra, and $\Delta_H$ and $\varepsilon_H$ are BiHom-algebra morphisms. 
Explicitly, for all $h,g\in H$,
\begin{align}
\Delta_H(hg)&=h_1g_1\otimes h_2g_2,
&
\Delta_H(1_H)&=1_H\otimes1_H, \label{eq:bihom-bialg-delta}\\
\varepsilon_H(hg)&=\varepsilon_H(h)\varepsilon_H(g),
&
\varepsilon_H(1_H)&=1_{\mathbb{k}}. \label{eq:bihom-bialg-eps}
\end{align}
\end{definition}

\begin{definition}
A \textit{monoidal BiHom-Hopf algebra} is a monoidal BiHom-bialgebra
\[
(H,\mu_H,\Delta_H,\alpha_H,\beta_H,1_H,\varepsilon_H)
\]
together with a linear map $S_H:H\to H$, called the \textit{antipode},
such that
\begin{align}
S_H\circ\alpha_H&=\alpha_H\circ S_H,
&
S_H\circ\beta_H&=\beta_H\circ S_H, \label{eq:antipode-commutes}\\
S_H(h_1)h_2&=\varepsilon_H(h)1_H
=
h_1S_H(h_2),\qquad h\in H. \label{eq:antipode}
\end{align}
\end{definition}

\begin{definition}
Let $(A,\mu_A,\alpha_A,\beta_A,1_A)$ be a monoidal BiHom-associative
algebra. A \textit{right $A$-BiHom-module} is a tuple $(M,\alpha_M,\beta_M,\cdot)$, where $M$ is a $\mathbb{k}$-vector space,
$\alpha_M,\beta_M\in\operatorname{Aut}_{\mathbb{k}}(M)$ are commuting
linear automorphisms, and $\cdot:M\otimes A\to M$ is a linear map,
$m\otimes a\mapsto m\cdot a$, such that, for all $m\in M$ and
$a,b\in A$,
\begin{align}
\alpha_M(m\cdot a)
&=\alpha_M(m)\cdot\alpha_A(a),
&
\beta_M(m\cdot a)
&=\beta_M(m)\cdot\beta_A(a), \label{eq:right-module-structure}\\
\alpha_M(m)\cdot(ab)
&=(m\cdot a)\cdot\beta_A(b), \label{eq:right-module-assoc}\\
m\cdot1_A&=\alpha_M(m). \label{eq:right-module-unit}
\end{align}
\end{definition}

A morphism of right $A$-BiHom-modules
\[
f:(M,\alpha_M,\beta_M)\longrightarrow(N,\alpha_N,\beta_N)
\]
is a linear map $f:M\to N$ satisfying
\[
f\circ\alpha_M=\alpha_N\circ f,\qquad
f\circ\beta_M=\beta_N\circ f, \qquad
\text{and}\qquad
f(m\cdot a)=f(m)\cdot a
\]
for all $m\in M$ and $a\in A$. The category of right $A$-BiHom-modules
is denoted by $\mathcal{M}_A$.

\begin{definition}
Let $(C,\Delta_C,\alpha_C,\beta_C,\varepsilon_C)$ be a monoidal
BiHom-coassociative coalgebra. A \textit{right $C$-BiHom-comodule} is a
tuple $(M,\alpha_M,\beta_M,\rho_M)$, where $M$ is a $\mathbb{k}$-vector space,
$\alpha_M,\beta_M\in\operatorname{Aut}_{\mathbb{k}}(M)$ are commuting
linear automorphisms, and $\rho_M:M\to M\otimes C$ is a linear map,
$\rho_M(m)=m_{[0]}\otimes m_{[1]}$, such that
\begin{align}
(\alpha_M\otimes\alpha_C)\circ\rho_M
&=\rho_M\circ\alpha_M,
&
(\beta_M\otimes\beta_C)\circ\rho_M
&=\rho_M\circ\beta_M, \label{eq:right-comodule-structure}\\
(\alpha_M^{-1}\otimes\Delta_C)\circ\rho_M
&=(\rho_M\otimes\beta_C^{-1})\circ\rho_M, \label{eq:right-comodule-coassoc}\\
(\operatorname{id}_M\otimes\varepsilon_C)\circ\rho_M
&=\alpha_M^{-1}. \label{eq:right-comodule-counit}
\end{align}
\end{definition}

A morphism of right $C$-BiHom-comodules
\[
f:(M,\alpha_M,\beta_M,\rho_M)\longrightarrow
(N,\alpha_N,\beta_N,\rho_N)
\]
is a linear map $f:M\to N$ satisfying
\[
f\circ\alpha_M=\alpha_N\circ f,\qquad
f\circ\beta_M=\beta_N\circ f,\qquad 
\text{and}\qquad 
(f\otimes\operatorname{id}_C)\circ\rho_M=\rho_N\circ f.
\]
The category of right $C$-BiHom-comodules is denoted by
$\mathcal{M}^C$.

\begin{definition}
Let $H$ be a monoidal BiHom-bialgebra and let
$(M,\alpha_M,\beta_M,\rho_M)$ be a right $H$-BiHom-comodule. The
\textit{space of coinvariants} of $M$ is
\[
M^{\operatorname{co}H}
=
\{\,m\in M\mid
\rho_M(m)=\alpha_M^{-1}(m)\otimes1_H\,\}.
\]
\end{definition}

\begin{definition}
Let $H$ be a monoidal BiHom-Hopf algebra. A \textit{right
$H$-BiHom-comodule algebra} is a monoidal BiHom-associative algebra
$(A,\mu_A,\alpha_A,\beta_A,1_A)$ equipped with a right
$H$-BiHom-comodule structure $\rho_A:A\to A\otimes H$, written
$\rho_A(a)=a_{[0]}\otimes a_{[1]}$, that is multiplicative and unital:
\begin{align}
\rho_A(ab)
&=a_{[0]}b_{[0]}\otimes a_{[1]}b_{[1]},
\qquad a,b\in A, \label{eq:comodule-algebra-mult}\\
\rho_A(1_A)&=1_A\otimes1_H. \label{eq:comodule-algebra-unit}
\end{align}
\end{definition}

\begin{lemma}\label{lem:coinvariant-subalgebra}
Let $A$ be a right $H$-BiHom-comodule algebra. Then
\[
B:=A^{\operatorname{co}H}
=
\{\,a\in A\mid \rho_A(a)=\alpha_A^{-1}(a)\otimes1_H\,\}
\]
is a monoidal BiHom-subalgebra of $A$. In particular, the restrictions
\[
\alpha_B:=\alpha_A|_B,\qquad \beta_B:=\beta_A|_B
\]
are well-defined automorphisms of $B$.
\end{lemma}

\begin{proof}
Clearly $1_A\in B$. Let $b,c\in B$. Then
\[
\rho_A(bc)
=
b_{[0]}c_{[0]}\otimes b_{[1]}c_{[1]}
=
\alpha_A^{-1}(b)\alpha_A^{-1}(c)\otimes1_H
=
\alpha_A^{-1}(bc)\otimes1_H,
\]
so $bc\in B$.

If $b\in B$, then
\[
\rho_A(\alpha_A(b))
=
(\alpha_A\otimes\alpha_H)\rho_A(b)
=
b\otimes1_H
=
\alpha_A^{-1}(\alpha_A(b))\otimes1_H,
\]
hence $\alpha_A(b)\in B$. Similarly,
\[
\rho_A(\beta_A(b))
=
(\beta_A\otimes\beta_H)\rho_A(b)
=
\beta_A\alpha_A^{-1}(b)\otimes1_H
=
\alpha_A^{-1}\beta_A(b)\otimes1_H,
\]
so $\beta_A(b)\in B$.

To see that the restrictions are automorphisms of $B$, let $b\in B$. Then
\[
\rho_A(\alpha_A^{-1}(b))
=
(\alpha_A^{-1}\otimes\alpha_H^{-1})\rho_A(b)
=
\alpha_A^{-2}(b)\otimes1_H
=
\alpha_A^{-1}(\alpha_A^{-1}(b))\otimes1_H,
\]
thus $\alpha_A^{-1}(b)\in B$. Likewise,
\[
\rho_A(\beta_A^{-1}(b))
=
(\beta_A^{-1}\otimes\beta_H^{-1})\rho_A(b)
=
\beta_A^{-1}\alpha_A^{-1}(b)\otimes1_H
=
\alpha_A^{-1}\beta_A^{-1}(b)\otimes1_H,
\]
so $\beta_A^{-1}(b)\in B$.

Therefore, $B$ is a monoidal BiHom-subalgebra of $A$, and $\alpha_A|_B$ and $\beta_A|_B$ are automorphisms of $B$.
\end{proof}

\begin{definition}
Let $H$ be a monoidal BiHom-Hopf algebra and let $A$ be a right
$H$-BiHom-comodule algebra. A \textit{right relative BiHom-Hopf module}
(also called a right $(H,A)$-BiHom-Doi--Hopf module in the special
relative sense used here) is a tuple
$(\!M\!,\!\alpha_M\!,\!\beta_M\!,\!\cdot\!,\!\rho_M\!)$ such that
$(M,\alpha_M,\beta_M,\cdot)$ is a right $A$-BiHom-module,
$(M,\alpha_M,\beta_M,\rho_M)$ is a right $H$-BiHom-comodule, and the
following compatibility condition holds:
\begin{equation}\label{eq:bihom-doi-hopf-compatibility}
\rho_M(m\cdot a)
=
m_{[0]}\cdot a_{[0]}\otimes m_{[1]}a_{[1]},
\qquad m\in M,\ a\in A.
\end{equation}
\end{definition}

A morphism of right relative BiHom-Hopf modules commutes with both
structure maps and is both right $A$-BiHom-linear and right
$H$-BiHom-colinear. These modules form the category $\mathcal{M}_A^H$.

\begin{lemma}\label{lem:coinvariants-right-B-module}
Let $M\in\mathcal{M}_A^H$. Then $M^{\operatorname{co}H}$ is a right
$A^{\operatorname{co}H}$-BiHom-module.
\end{lemma}

\begin{proof}
Let $m\in M^{\operatorname{co}H}$ and $b\in A^{\operatorname{co}H}$. Then
\[
\rho_M(m\cdot b)
=
m_{[0]}\cdot b_{[0]}\otimes m_{[1]}b_{[1]}
=
\alpha_M^{-1}(m)\cdot\alpha_A^{-1}(b)\otimes1_H.
\]
Since
\[
\alpha_M^{-1}(m\cdot b)=\alpha_M^{-1}(m)\cdot\alpha_A^{-1}(b),
\]
we get
\[
\rho_M(m\cdot b)=\alpha_M^{-1}(m\cdot b)\otimes1_H,
\]
hence $m\cdot b\in M^{\operatorname{co}H}$.

If $m\in M^{\operatorname{co}H}$, then
\[
\rho_M(\alpha_M(m))
=
(\alpha_M\otimes\alpha_H)\rho_M(m)
=
m\otimes1_H
=
\alpha_M^{-1}(\alpha_M(m))\otimes1_H,
\]
so $\alpha_M(m)\in M^{\operatorname{co}H}$. Similarly,
\[
\rho_M(\beta_M(m))
=
(\beta_M\otimes\beta_H)\rho_M(m)
=
\beta_M\alpha_M^{-1}(m)\otimes1_H
=
\alpha_M^{-1}\beta_M(m)\otimes1_H,
\]
hence $\beta_M(m)\in M^{\operatorname{co}H}$.

Moreover,
\[
\rho_M(\alpha_M^{-1}(m))
=
(\alpha_M^{-1}\otimes\alpha_H^{-1})\rho_M(m)
=
\alpha_M^{-2}(m)\otimes1_H
=
\alpha_M^{-1}(\alpha_M^{-1}(m))\otimes1_H,
\]
and
\[
\rho_M(\beta_M^{-1}(m))
=
(\beta_M^{-1}\otimes\beta_H^{-1})\rho_M(m)
=
\beta_M^{-1}\alpha_M^{-1}(m)\otimes1_H
=
\alpha_M^{-1}\beta_M^{-1}(m)\otimes1_H,
\]
so both $\alpha_M^{-1}$ and $\beta_M^{-1}$ preserve
$M^{\operatorname{co}H}$.

Therefore $M^{\operatorname{co}H}$ is a right
$A^{\operatorname{co}H}$-BiHom-module.
\end{proof}

\begin{definition}
Let $(H,\mu_H,\Delta_H,\alpha_H,\beta_H,1_H,\varepsilon_H,S_H)$ be a monoidal BiHom-Hopf algebra. A linear functional
$\lambda\in H^*$ is called a \textit{left integral} if
\begin{align}
h_1\lambda(h_2)&=\lambda(h)1_H,\quad (h\in H), \qquad \text{and}\qquad
\lambda\circ\alpha_H=\lambda, \,\,\,\,\lambda\circ\beta_H=\lambda. \label{eq:integral-invariant}
\end{align}
It is called a \textit{right integral} if
\begin{align}
\lambda(h_1)h_2&=\lambda(h)1_H,\quad (h\in H),\qquad \text{and}\qquad
\lambda\circ\alpha_H=\lambda, \,\,\,\,\lambda\circ\beta_H=\lambda.
\label{eq:right-integral-invariant}
\end{align}
A left or right integral is \textit{normalized} if
$\lambda(1_H)=1_{\mathbb{k}}$. A \textit{Haar integral} is a normalized
functional that is both a left and a right integral.
\end{definition}

\begin{lemma}\label{lem:categories-abelian}
The categories $\mathcal{M}_A$, $\mathcal{M}^H$, and
$\mathcal{M}_A^H$ are abelian. 
\end{lemma}

\begin{proof}
All three categories are $\mathbb{k}$-linear and have finite biproducts,
formed on the underlying vector spaces.  Let $f:M\to N$ be a morphism
in any one of these categories.  Because $f$ commutes with each
structure automorphism and its inverse, $\ker f$ and $\operatorname{im}f$
are stable under $\alpha^{\pm1}$ and $\beta^{\pm1}$, while the same maps
induce automorphisms on $\operatorname{coker}f$.  The module action and
the comodule coaction restrict to the kernel and image and descend to
the cokernel.  This follows directly from $A$-linearity and
$H$-colinearity; for the coaction, one uses the exactness of tensoring
with the $\mathbb{k}$-vector space $H$.  The BiHom module, comodule, and
relative compatibility identities are inherited by these subspaces and
quotients.  Thus kernels and cokernels are created by the forgetful
functor to vector spaces.  In particular, monomorphisms and
epimorphisms are exactly the morphisms whose underlying linear maps are
injective and surjective, respectively.  The canonical map from the coimage of $f$ to
its image is consequently the usual vector-space isomorphism, and it
preserves all the displayed structures.  Hence each category is
abelian.
\end{proof}

\begin{definition}
Let $\mathcal{C}$ be an abelian category. A \textit{torsion theory} in
$\mathcal{C}$ is a pair $(\mathcal{U},\mathcal{V})$ of full replete
subcategories satisfying the following conditions:
\begin{enumerate}
\item For every object $X\in\mathcal{C}$, there exists a short exact
sequence
\[
0\longrightarrow T\longrightarrow X\longrightarrow F\longrightarrow0
\]
with $T\in\mathcal{U}$ and $F\in\mathcal{V}$.
\item For all $T\in\mathcal{U}$ and $F\in\mathcal{V}$, every morphism
$T\to F$ is zero.
\end{enumerate}
The subcategory $\mathcal{U}$ is called the \textit{torsion class}, and
$\mathcal{V}$ is called the \textit{torsion-free class}. The torsion
theory is called \textit{hereditary} if $\mathcal{U}$ is closed under
subobjects.
\end{definition}

\section{Induction functors and torsion theory}

\subsection{Induction functors for relative BiHom-Hopf modules}

Throughout Sections~3 and~4, we fix a monoidal BiHom-Hopf algebra
$H$, a right $H$-BiHom-comodule algebra $A$, and the coinvariant
subalgebra $B=A^{\operatorname{co}H}$.  In this subsection, we construct
the induction functor and prove that it is left adjoint to the
coinvariant functor.

\begin{definition}\label{def:balanced-tensor-product}
Let $(N,\alpha_N,\beta_N)$ be a right $B$-BiHom-module. Let $R$ be the
linear subspace of $N\otimes A$ generated by all elements
\[
(n\cdot b)\otimes \beta_A(a)-\alpha_N(n)\otimes ba,
\qquad n\in N,\ b\in B,\ a\in A.
\]
The \textit{right balanced tensor product} of $N$ and $A$ over $B$ is
defined by
\[
N\otimes_B A=(N\otimes A)/R.
\]
The class of $n\otimes a$ in $N\otimes_B A$ will be denoted by
$n\otimes_B a$.
Thus, in $N\otimes_B A$ we have
\begin{equation}\label{eq:balanced-relation}
(n\cdot b)\otimes_B\beta_A(a)
=
\alpha_N(n)\otimes_B ba,
\qquad n\in N,\ b\in B,\ a\in A.
\end{equation}
Equivalently, since $\beta_A$ is bijective,
\begin{equation}\label{eq:balanced-relation-equivalent}
(n\cdot b)\otimes_B a
=
\alpha_N(n)\otimes_B b\beta_A^{-1}(a),
\qquad n\in N,\ b\in B,\ a\in A.
\end{equation}
\end{definition}

\begin{remark}\label{rem:balanced-universal-property}
The terminology ``balanced tensor product'' refers to the following
universal property.  If $V$ is a vector space and
$t:N\times A\to V$ is bilinear with
\[
t(n\cdot b,\beta_A(a))=t(\alpha_N(n),ba)
\qquad(n\in N,\ b\in B,\ a\in A),
\]
then there is a unique linear map
$\widetilde t:N\otimes_BA\to V$ satisfying
$\widetilde t(n\otimes_Ba)=t(n,a)$.  This is precisely the universal
property of the quotient by $R$.
\end{remark}

\begin{lemma}\label{lem:balanced-structure-maps}
The linear maps $\alpha_N\otimes\alpha_A$ and $\beta_N\otimes\beta_A$
on $N\otimes A$ induce well-defined automorphisms of $N\otimes_B A$.
\end{lemma}

\begin{proof}
It is enough to show that the defining subspace $R$ is stable under $\alpha_N^{\pm1}\otimes\alpha_A^{\pm1}$ and $\beta_N^{\pm1}\otimes\beta_A^{\pm1}$.

For the map $\alpha_N\otimes\alpha_A$, take a generator $r=(n\cdot b)\otimes\beta_A(a)-\alpha_N(n)\otimes ba$. Then
\[
(\alpha_N\otimes\alpha_A)(r)
=
(\alpha_N(n)\cdot\alpha_B(b))\otimes\beta_A(\alpha_A(a))
-\alpha_N^2(n)\otimes\alpha_B(b)\alpha_A(a),
\]
which is again a generator of $R$.

For the inverse map $\alpha_N^{-1}\otimes\alpha_A^{-1}$, we get
\[
(\alpha_N^{-1}\otimes\alpha_A^{-1})(r)
=
(\alpha_N^{-1}(n)\cdot\alpha_B^{-1}(b))\otimes\beta_A(\alpha_A^{-1}(a))
-n\otimes \alpha_B^{-1}(b)\alpha_A^{-1}(a),
\]
again a generator of $R$.

The verification for $\beta_N^{\pm1}\otimes\beta_A^{\pm1}$ is analogous.
Therefore, $\alpha_N\otimes\alpha_A$ and $\beta_N\otimes\beta_A$ descend
to well-defined automorphisms of $N\otimes_BA$.
\end{proof}

Henceforth, write
$\alpha_{N\otimes_B A}:=\alpha_N\otimes\alpha_A$ and
$\beta_{N\otimes_B A}:=\beta_N\otimes\beta_A$.

Let $N\in\mathcal{M}_B$. On $N\otimes_B A$ we define
\begin{align}
\alpha_{F(N)}(n\otimes_B a)
&=\alpha_N(n)\otimes_B\alpha_A(a), \label{eq:F-alpha}\\
\beta_{F(N)}(n\otimes_B a)
&=\beta_N(n)\otimes_B\beta_A(a), \label{eq:F-beta}\\
(n\otimes_B a)\cdot c
&=\alpha_N(n)\otimes_B a\beta_A^{-1}(c),
\qquad c\in A, \label{eq:F-action}\\
\rho_{F(N)}(n\otimes_B a)
&=
\alpha_N^{-1}(n)\otimes_B a_{[0]}\otimes\beta_H(a_{[1]}).
\label{eq:F-coaction}
\end{align}

\begin{proposition}\label{prop:induction-object-well-defined}

$N\otimes_B A$ is a right relative BiHom-Hopf module over $(H,A)$.
\end{proposition}

\begin{proof}
We divide the proof into several verifications.
By Lemma~\ref{lem:balanced-structure-maps}, the two induced structure
maps are commuting automorphisms.

First, the right $A$-action is well-defined. For a generator of the
balanced relation, we compute, for $c\in A$,
\begin{align*}
\bigl((n\cdot b)\otimes_B\beta_A(a)\bigr)\cdot c
&=
\alpha_N(n\cdot b)\otimes_B \beta_A(a)\beta_A^{-1}(c)       \\
&=
\alpha_N(n)\cdot\alpha_B(b)\otimes_B
\beta_A\bigl(a\beta_A^{-2}(c)\bigr)                         \\
&=
\alpha_N^2(n)\otimes_B
\alpha_B(b)\bigl(a\beta_A^{-2}(c)\bigr)                      \\
&=
\alpha_N^2(n)\otimes_B (ba)\beta_A^{-1}(c)                   \\
&=
\bigl(\alpha_N(n)\otimes_B ba\bigr)\cdot c.
\end{align*}
In the third equality, we used the balancing relation, and in the fourth
one, the BiHom-associativity identity $\alpha_A(b)(a\beta_A^{-2}(c))=(ba)\beta_A^{-1}(c)$. Thus, the action descends to $N\otimes_B A$.

The BiHom-module identities follow directly from
Definition~\ref{def:balanced-tensor-product} and the BiHom-associativity
of $A$.  Compatibility with the two structure maps is explicit:
\begin{align*}
\alpha_{F(N)}\bigl((n\otimes_Ba)\cdot c\bigr)
&=\alpha_N^2(n)\otimes_B
  \alpha_A(a)\alpha_A\beta_A^{-1}(c)\\
&=\alpha_{F(N)}(n\otimes_Ba)\cdot\alpha_A(c),\\
\beta_{F(N)}\bigl((n\otimes_Ba)\cdot c\bigr)
&=\alpha_N\beta_N(n)\otimes_B\beta_A(a)c\\
&=\beta_{F(N)}(n\otimes_Ba)\cdot\beta_A(c).
\end{align*}
For BiHom-associativity of the action, we compute
\begin{align*}
\alpha_{F(N)}(n\otimes_B a)\cdot (cd)
&=
\alpha_N^2(n)\otimes_B
\alpha_A(a)\beta_A^{-1}(cd)                                  \\
&=
\alpha_N^2(n)\otimes_B
\alpha_A(a)\bigl(\beta_A^{-1}(c)\beta_A^{-1}(d)\bigr),
\end{align*}
whereas
\begin{align*}
\bigl((n\otimes_B a)\cdot c\bigr)\cdot\beta_A(d)
&=
\bigl(\alpha_N(n)\otimes_B a\beta_A^{-1}(c)\bigr)\cdot
\beta_A(d)                                                    \\
&=
\alpha_N^2(n)\otimes_B
\bigl(a\beta_A^{-1}(c)\bigr)d.
\end{align*}
These two elements are equal by the BiHom-associativity of $A$. The unit
identity is
\[
(n\otimes_B a)\cdot1_A
=
\alpha_N(n)\otimes_B a1_A
=
\alpha_N(n)\otimes_B\alpha_A(a)
=
\alpha_{F(N)}(n\otimes_B a).
\]

Second, the right $H$-coaction is well-defined. We check that it respects
the balancing relation. Since $b\in B$, we have $\rho_A(b)=\alpha_A^{-1}(b)\otimes1_H$. 
Then
\begin{align*}
\rho_{F(N)}\bigl((n\cdot b)\otimes_B\beta_A(a)\bigr)
&=
\alpha_N^{-1}(n\cdot b)\otimes_B
\beta_A(a_{[0]})\otimes\beta_H^2(a_{[1]})                    \\
&=
\bigl(\alpha_N^{-1}(n)\cdot\alpha_B^{-1}(b)\bigr)\otimes_B
\beta_A(a_{[0]})\otimes\beta_H^2(a_{[1]})                    \\
&=
n\otimes_B \alpha_B^{-1}(b)a_{[0]}\otimes\beta_H^2(a_{[1]}).
\end{align*}
On the other hand,
\begin{align*}
\rho_{F(N)}\bigl(\alpha_N(n)\otimes_B ba\bigr)
&=
n\otimes_B (ba)_{[0]}\otimes\beta_H((ba)_{[1]})               \\
&=
n\otimes_B \alpha_B^{-1}(b)a_{[0]}
   \otimes\beta_H(1_Ha_{[1]})                                 \\
&=
n\otimes_B \alpha_B^{-1}(b)a_{[0]}
   \otimes\beta_H^2(a_{[1]}).
\end{align*}
Here, we used the comodule algebra identity and the BiHom-unit identity
$1_Hh=\beta_H(h)$. Hence, the coaction is well-defined.

The coaction is compatible with both structure maps.  Indeed,
\begin{align*}
\rho_{F(N)}\alpha_{F(N)}(n\otimes_Ba)
&=n\otimes_B\alpha_A(a_{[0]})
  \otimes\alpha_H\beta_H(a_{[1]})\\
&=(\alpha_{F(N)}\otimes\alpha_H)
  \rho_{F(N)}(n\otimes_Ba),\\
\rho_{F(N)}\beta_{F(N)}(n\otimes_Ba)
&=\alpha_N^{-1}\beta_N(n)\otimes_B\beta_A(a_{[0]})
  \otimes\beta_H^2(a_{[1]})\\
&=(\beta_{F(N)}\otimes\beta_H)
  \rho_{F(N)}(n\otimes_Ba).
\end{align*}
We next spell out the coassociativity condition. For
$n\in N$ and $a\in A$,
\begin{align*}
(\alpha_{F(N)}^{-1}\otimes\Delta_H)\rho_{F(N)}(n\otimes_Ba)
&=
\alpha_N^{-2}(n)\otimes_B\alpha_A^{-1}(a_{[0]})
\otimes\beta_H(a_{[1]1})\otimes\beta_H(a_{[1]2}),
\end{align*}
whereas $(\rho_{F(N)}\otimes\beta_H^{-1})\rho_{F(N)}(n\otimes_Ba)= \alpha_N^{-2}(n)\otimes_B a_{[0][0]} \otimes\beta_H(a_{[0][1]})\otimes a_{[1]}$.
These two expressions are equal because the right $H$-BiHom-comodule
structure of $A$ satisfies $(\alpha_A^{-1}\otimes\Delta_H)\rho_A = (\rho_A\otimes\beta_H^{-1})\rho_A$, after applying $\operatorname{id}\otimes\beta_H\otimes\beta_H$ to the last two tensor factors. The counit identity follows from $(\operatorname{id}_A\otimes\varepsilon_H)\rho_A=\alpha_A^{-1}$.
More explicitly,
\[
(\operatorname{id}_{F(N)}\otimes\varepsilon_H)
\rho_{F(N)}(n\otimes_Ba)
=\alpha_N^{-1}(n)\otimes_B\alpha_A^{-1}(a)
=\alpha_{F(N)}^{-1}(n\otimes_Ba).
\]

Finally, we verify the relative compatibility condition. For $c\in A$,
\begin{align*}
\rho_{F(N)}\bigl((n\otimes_Ba)\cdot c\bigr)
&=
\rho_{F(N)}\bigl(\alpha_N(n)\otimes_B a\beta_A^{-1}(c)\bigr)  \\
&=
n\otimes_B
\bigl(a\beta_A^{-1}(c)\bigr)_{[0]}
\otimes
\beta_H\bigl((a\beta_A^{-1}(c))_{[1]}\bigr)                  \\
&=
n\otimes_B
a_{[0]}\beta_A^{-1}(c_{[0]})
\otimes
\beta_H(a_{[1]})c_{[1]}.
\end{align*}
On the other hand,
\begin{align*}
(n\otimes_Ba)_{[0]}\cdot c_{[0]}
\otimes
(n\otimes_Ba)_{[1]}c_{[1]}
&=
\bigl(\alpha_N^{-1}(n)\otimes_Ba_{[0]}\bigr)\cdot c_{[0]}
\otimes
\beta_H(a_{[1]})c_{[1]}                                      \\
&=
n\otimes_Ba_{[0]}\beta_A^{-1}(c_{[0]})
\otimes
\beta_H(a_{[1]})c_{[1]}.
\end{align*}
Thus, the compatibility condition holds. Hence $N\otimes_BA$ is an object
of $\mathcal{M}_A^H$.
\end{proof}

The \textit{induced functor}
\[
F=-\otimes_BA:\mathcal{M}_B\longrightarrow\mathcal{M}_A^H
\]
is defined as follows: $F(N)=N\otimes_BA$ and if
$f:N\to N'$ is a morphism in $\mathcal{M}_B$, we define
\[
F(f):N\otimes_BA\longrightarrow N'\otimes_BA,
\qquad
F(f)(n\otimes_Ba)=f(n)\otimes_Ba.
\]

\begin{proposition}\label{prop:induction-functor}
The assignment $F:\mathcal{M}_B\longrightarrow\mathcal{M}_A^H$ is a well-defined functor.
\end{proposition}

\begin{proof}
We first check that $F(f)$ is well-defined. If $(n\cdot b)\otimes_B\beta_A(a)=\alpha_N(n)\otimes_Bba$, then
\begin{align*}
F(f)\bigl((n\cdot b)\otimes_B\beta_A(a)\bigr)
&=
f(n\cdot b)\otimes_B\beta_A(a)                                \\
&=
f(n)\cdot b\otimes_B\beta_A(a)                                 \\
&=
\alpha_{N'}(f(n))\otimes_Bba                                   \\
&=
F(f)\bigl(\alpha_N(n)\otimes_Bba\bigr).
\end{align*}
Thus $F(f)$ is well-defined. Because $f$ commutes with the two
structure maps, so does $F(f)$. It is right $A$-linear, because
\begin{align*}
F(f)\bigl((n\otimes_Ba)\cdot c\bigr)
&=
F(f)\bigl(\alpha_N(n)\otimes_Ba\beta_A^{-1}(c)\bigr)            \\
&=
\alpha_{N'}(f(n))\otimes_Ba\beta_A^{-1}(c)                      \\
&=
F(f)(n\otimes_Ba)\cdot c.
\end{align*}
It is right $H$-colinear, since
\begin{align*}
\rho_{F(N')}\bigl(F(f)(n\otimes_Ba)\bigr)
&=
\rho_{F(N')}(f(n)\otimes_Ba)                                   \\
&=
\alpha_{N'}^{-1}(f(n))\otimes_Ba_{[0]}\otimes\beta_H(a_{[1]})   \\
&=
f(\alpha_N^{-1}(n))\otimes_Ba_{[0]}\otimes\beta_H(a_{[1]})      \\
&=
(F(f)\otimes\operatorname{id}_H)\rho_{F(N)}(n\otimes_Ba).
\end{align*}
Therefore $F(f)$ is a morphism in $\mathcal{M}_A^H$. Functoriality,
namely $F(\operatorname{id}_N)=\operatorname{id}_{F(N)}$ and
$F(gf)=F(g)F(f)$, is immediate from the definition.
\end{proof}

The \textit{coinvariant functor}
\[
G=(-)^{\operatorname{co}H}:\mathcal{M}_A^H\longrightarrow\mathcal{M}_B
\]
is defined by
\[
G(M)=M^{\operatorname{co}H}
=
\{\,m\in M\mid \rho_M(m)=\alpha_M^{-1}(m)\otimes1_H\,\}
\quad \text{and} \quad
G(f)=f|_{M^{\operatorname{co}H}}.
\]

Lemma~\ref{lem:coinvariants-right-B-module} shows that $G(M)$ is an
object of $\mathcal M_B$.  Colinearity sends coinvariants to
coinvariants, and the restriction of an $A$-linear morphism is
$B$-linear; hence the displayed assignment is a well-defined functor.

\begin{theorem}\label{thm:induction-coinvariant-adjunction}
The induced functor $F$ is left adjoint to the coinvariant functor $G$.
\end{theorem}

\begin{proof}
We define the unit and counit as follows. For $N\in\mathcal{M}_B$, let
\[
\eta_N:N\longrightarrow (N\otimes_BA)^{\operatorname{co}H},
\qquad
\eta_N(n)=\alpha_N^{-1}(n)\otimes_B1_A.
\]
For $M\in\mathcal{M}_A^H$, let
\[
\delta_M:M^{\operatorname{co}H}\otimes_BA\longrightarrow M,
\qquad
\delta_M(m\otimes_Ba)=m\cdot a.
\]

First, $\eta_N$ is well-defined. Indeed,
\begin{align*}
\rho_{F(N)}(\alpha_N^{-1}(n)\otimes_B1_A)
&=
\alpha_N^{-2}(n)\otimes_B1_A\otimes1_H                     \\
&=
\alpha_{F(N)}^{-1}\bigl(\alpha_N^{-1}(n)\otimes_B1_A\bigr)
\otimes1_H.
\end{align*}
Hence $\eta_N(n)\in (N\otimes_BA)^{\operatorname{co}H}$.

It commutes with the structure maps, since
\[
\eta_N(\alpha_N(n))=n\otimes_B1_A
=\alpha_{F(N)}(\eta_N(n))
\]
and
\[
\eta_N(\beta_N(n))
=\alpha_N^{-1}\beta_N(n)\otimes_B1_A
=\beta_{F(N)}(\eta_N(n)).
\]

Moreover, $\eta_N$ is right $B$-linear. For $b\in B$,
\begin{align*}
\eta_N(n\cdot b)
&=
\alpha_N^{-1}(n\cdot b)\otimes_B1_A                         \\
&=
\bigl(\alpha_N^{-1}(n)\cdot\alpha_B^{-1}(b)\bigr)\otimes_B1_A\\
&=
n\otimes_B\alpha_B^{-1}(b)1_A                                \\
&=
n\otimes_Bb.
\end{align*}
On the other hand,
\[
\eta_N(n)\cdot b
=
(\alpha_N^{-1}(n)\otimes_B1_A)\cdot b
=
n\otimes_B1_A\beta_A^{-1}(b)
=
n\otimes_Bb.
\]
Thus $\eta_N(n\cdot b)=\eta_N(n)\cdot b$.

Next, $\delta_M$ is well-defined with respect to the balancing relation.
For $m\in M^{\operatorname{co}H}$, $b\in B$ and $a\in A$,
\begin{align*}
\delta_M\bigl((m\cdot b)\otimes_B\beta_A(a)\bigr)
&=
(m\cdot b)\cdot\beta_A(a)                                    \\
&=
\alpha_M(m)\cdot(ba)                                         \\
&=
\delta_M\bigl(\alpha_M(m)\otimes_Bba\bigr).
\end{align*}
It commutes with the structure maps because
\begin{align*}
\delta_M(\alpha_M(m)\otimes_B\alpha_A(a))
&=\alpha_M(m\cdot a),\\
\delta_M(\beta_M(m)\otimes_B\beta_A(a))
&=\beta_M(m\cdot a).
\end{align*}
The map $\delta_M$ is right $A$-linear because
\begin{align*}
\delta_M\bigl((m\otimes_Ba)\cdot c\bigr)
&=
\delta_M\bigl(\alpha_M(m)\otimes_Ba\beta_A^{-1}(c)\bigr)      \\
&=
\alpha_M(m)\cdot(a\beta_A^{-1}(c))                            \\
&=
(m\cdot a)\cdot c                                             \\
&=
\delta_M(m\otimes_Ba)\cdot c.
\end{align*}
It is also right $H$-colinear. Since $m\in M^{\operatorname{co}H}$, $\rho_M(m)=\alpha_M^{-1}(m)\otimes1_H$.
Therefore
\begin{align*}
\rho_M(\delta_M(m\otimes_Ba))
&=
\rho_M(m\cdot a)                                              \\
&=
m_{[0]}\cdot a_{[0]}\otimes m_{[1]}a_{[1]}                    \\
&=
\alpha_M^{-1}(m)\cdot a_{[0]}\otimes1_Ha_{[1]}                \\
&=
\alpha_M^{-1}(m)\cdot a_{[0]}\otimes\beta_H(a_{[1]})          \\
&=
(\delta_M\otimes\operatorname{id}_H)
\rho_{F(G(M))}(m\otimes_Ba).
\end{align*}
Thus $\delta_M$ is a morphism in $\mathcal{M}_A^H$.

If $u:N\to N'$ is a morphism in $\mathcal M_B$ and $v:M\to M'$ is a
morphism in $\mathcal M_A^H$, their defining formulas give
$G(F(u))\eta_N=\eta_{N'}u$ and
$v\delta_M=\delta_{M'}F(G(v))$.  Thus $\eta$ and $\delta$ are natural.
It remains to verify the triangular identities.
For $n\in N$ and $a\in A$,
\begin{align*}
\delta_{F(N)}\bigl(F(\eta_N)(n\otimes_Ba)\bigr)
&=
\delta_{F(N)}\bigl((\alpha_N^{-1}(n)\otimes_B1_A)\otimes_Ba\bigr)\\
&=
(\alpha_N^{-1}(n)\otimes_B1_A)\cdot a                         \\
&=
n\otimes_B1_A\beta_A^{-1}(a)                                  \\
&=
n\otimes_Ba.
\end{align*}
Here we used the BiHom-unit identity $1_Ax=\beta_A(x)$.
For $m\in M^{\operatorname{co}H}$,
\begin{align*}
G(\delta_M)\bigl(\eta_{G(M)}(m)\bigr)
&=
G(\delta_M)(\alpha_M^{-1}(m)\otimes_B1_A)                     \\
&=
\alpha_M^{-1}(m)\cdot1_A                                      \\
&=
m.
\end{align*}
Hence $F\dashv G$.
\end{proof}

\subsection{Torsion theory for relative BiHom-Hopf modules}

For the remainder of the paper, assume additionally that $H$ has a fixed
Haar integral $\lambda\in H^*$. We use it to construct a torsion theory
on $\mathcal{M}_A^H$.

Define
\[
x*y=\alpha_H^{-1}(x)\beta_H^{-1}(y),
\qquad
\Delta^\sharp(h)=\alpha_H(h_1)\otimes\beta_H(h_2).
\]
When using Sweedler notation for the untwisted coproduct, we write
$\Delta^\sharp(h)=h_1^\sharp\otimes h_2^\sharp$; hence
$h_1^\sharp=\alpha_H(h_1)$ and $h_2^\sharp=\beta_H(h_2)$ inside
Sweedler-suppressed expressions.
This is the usual untwisting of a BiHom structure with bijective
structure maps; compare \cite{GrazianiMakhloufMeniniPanaite2015}.  We
include the verification because the transported integral identity used
below depends on the precise conventions.

\begin{lemma}\label{lem:untwist-bihom-hopf}
Let $(H,\mu_H,\Delta_H,\alpha_H,\beta_H,1_H,\varepsilon_H,S_H)$ be a monoidal BiHom-Hopf algebra with bijective structure maps. 
Then $H^\sharp=(H,*,\Delta^\sharp,1_H,\varepsilon_H,S_H)$ is an ordinary Hopf algebra.
\end{lemma}

\begin{proof}
We first check associativity of $*$. For $x,y,z\in H$,
\[
(x*y)*z
=
\alpha_H^{-1}(x*y)\beta_H^{-1}(z)
=
(\alpha_H^{-2}(x)\alpha_H^{-1}\beta_H^{-1}(y))\beta_H^{-1}(z),
\]
while
\[
x*(y*z)
=
\alpha_H^{-1}(x)\beta_H^{-1}(y*z)
=
\alpha_H^{-1}(x)\bigl(\alpha_H^{-1}\beta_H^{-1}(y)\beta_H^{-2}(z)\bigr).
\]
Applying the BiHom-associativity of $H$ to $a=\alpha_H^{-2}(x)$, $b=\alpha_H^{-1}\beta_H^{-1}(y)$ and $c=\beta_H^{-2}(z)$, we obtain
\[
\alpha_H^{-1}(x)\bigl(\alpha_H^{-1}\beta_H^{-1}(y)\beta_H^{-2}(z)\bigr)
=
(\alpha_H^{-2}(x)\alpha_H^{-1}\beta_H^{-1}(y))\beta_H^{-1}(z).
\]
Hence $*$ is associative.

The unit is $1_H$, because $x*1_H=\alpha_H^{-1}(x)1_H=x$ and $1_H*x=1_H\beta_H^{-1}(x)=x$, by the BiHom-unit identities.

Next, we check coassociativity of $\Delta^\sharp$. We compute
\[
(\Delta^\sharp\otimes\operatorname{id})\Delta^\sharp(h)
=
\alpha_H^2(h_{11})\otimes\alpha_H\beta_H(h_{12})\otimes\beta_H(h_2),
\]
and
\[
(\operatorname{id}\otimes\Delta^\sharp)\Delta^\sharp(h)
=
\alpha_H(h_1)\otimes\alpha_H\beta_H(h_{21})\otimes\beta_H^2(h_{22}).
\]
These are equal because the BiHom-coassociativity identity 
$(\Delta_H\otimes\beta_H^{-1})\Delta_H = (\alpha_H^{-1}\otimes\Delta_H)\Delta_H$
becomes exactly the above equality after applying $\alpha_H^2\otimes\alpha_H\beta_H\otimes\beta_H^2$.

The counit property follows from
\[
(\varepsilon_H\otimes\operatorname{id})\Delta^\sharp(h)
=
\varepsilon_H(\alpha_H(h_1))\beta_H(h_2)
=
\varepsilon_H(h_1)\beta_H(h_2)
=
h
\]
and
\[
(\operatorname{id}\otimes\varepsilon_H)\Delta^\sharp(h)
=
\alpha_H(h_1)\varepsilon_H(\beta_H(h_2))
=
\alpha_H(h_1)\varepsilon_H(h_2)
=
h.
\]

We now verify that $\Delta^\sharp$ and $\varepsilon_H$ are algebra maps.
Using multiplicativity of $\Delta_H$, we have
\begin{align*}
\Delta^\sharp(x*y)
&=
\Delta^\sharp(\alpha_H^{-1}(x)\beta_H^{-1}(y)) \\
&=
\alpha_H(\alpha_H^{-1}(x_1)\beta_H^{-1}(y_1))
\otimes
\beta_H(\alpha_H^{-1}(x_2)\beta_H^{-1}(y_2)) \\
&=
x_1\,\alpha_H\beta_H^{-1}(y_1)
\otimes
\beta_H\alpha_H^{-1}(x_2)\,y_2.
\end{align*}
On the other hand,
\begin{align*}
\Delta^\sharp(x)\Delta^\sharp(y)
&=
(\alpha_H(x_1)*\alpha_H(y_1))
\otimes
(\beta_H(x_2)*\beta_H(y_2)) \\
&=
x_1\,\beta_H^{-1}\alpha_H(y_1)
\otimes
\alpha_H^{-1}\beta_H(x_2)\,y_2 \\
&=
x_1\,\alpha_H\beta_H^{-1}(y_1)
\otimes
\beta_H\alpha_H^{-1}(x_2)\,y_2.
\end{align*}
Thus $\Delta^\sharp(x*y)=\Delta^\sharp(x)\Delta^\sharp(y)$. Also
\[
\Delta^\sharp(1_H)=1_H\otimes1_H,
\qquad
\varepsilon_H(x*y)=\varepsilon_H(x)\varepsilon_H(y),
\qquad
\varepsilon_H(1_H)=1.
\]

Finally, for the antipode, since $\Delta^\sharp(h)=\alpha_H(h_1)\otimes\beta_H(h_2)$, we get
\[
S_H(h_1^\sharp)*h_2^\sharp = S_H(\alpha_H(h_1))*\beta_H(h_2) = S_H(h_1)h_2 = \varepsilon_H(h)1_H,
\]
and similarly
\[
h_1^\sharp*S_H(h_2^\sharp)
=
\alpha_H(h_1)*S_H(\beta_H(h_2))
=
h_1S_H(h_2)
=
\varepsilon_H(h)1_H.
\]
Therefore $H^\sharp$ is an ordinary Hopf algebra.
\end{proof}

\begin{lemma}\label{lem:bihom-haar-identities}
Let $(H,\mu_H,\Delta_H,\alpha_H,\beta_H,1_H,\varepsilon_H,S_H)$ be a monoidal BiHom-Hopf algebra, and let
$\lambda\in H^*$ be a Haar integral. Then, for all $h,g\in H$,
\begin{equation}\label{eq:transported-haar}
\beta_H(h_2)\lambda(h_1g)
=
\lambda\bigl(\alpha_H^{-1}(h)\alpha_H(g_1)\bigr)\,
\beta_H^2(S_H(g_2)).
\end{equation}
Here, juxtaposition denotes the BiHom multiplication of $H$.
\end{lemma}

\begin{proof}
By Lemma~\ref{lem:untwist-bihom-hopf}, the same underlying vector space
carries an ordinary Hopf algebra structure. We first observe that $\lambda$ is a normalized two-sided integral on
$H^\sharp$. Indeed,
\[
h^\sharp_1\lambda(h^\sharp_2)
=
\alpha_H(h_1)\lambda(\beta_H(h_2))
=
\alpha_H(h_1)\lambda(h_2)
=
\alpha_H\bigl(h_1\lambda(h_2)\bigr)
=
\lambda(h)1_H,
\]
where we used $\lambda\circ\beta_H=\lambda$ and the left integral
identity in $H$. Similarly,
\[
\lambda(h^\sharp_1)h^\sharp_2
=
\lambda(\alpha_H(h_1))\beta_H(h_2)
=
\lambda(h_1)\beta_H(h_2)
=
\beta_H\bigl(\lambda(h_1)h_2\bigr)
=
\lambda(h)1_H.
\]
Hence $\lambda$ is a Haar integral on the ordinary Hopf algebra
$H^\sharp$.

We use the following identity for an ordinary Hopf algebra with a
two-sided integral $\lambda$:
\begin{equation}\label{eq:ordinary-haar-used}
x^\sharp_2\,\lambda(x^\sharp_1*y)
=
\lambda(x*y^\sharp_1)\,S_H(y^\sharp_2).
\end{equation}
Indeed, suppressing the superscript $\sharp$ for this three-line
calculation and using the right-integral identity for $xy_1$, we have
\begin{align*}
x_2\lambda(x_1y)
&=x_2\lambda(x_1y_1)\varepsilon(y_2)\\
&=\lambda(x_1y_1)x_2y_2S_H(y_3)\\
&=\lambda(xy_1)S_H(y_2).
\end{align*}
This proves \eqref{eq:ordinary-haar-used} and fixes the left/right
integral convention used here.

Now apply \eqref{eq:ordinary-haar-used} in $H^\sharp$ with $x=h$ and $y=\beta_H(g)$.
Since $\Delta^\sharp(h)=\alpha_H(h_1)\otimes\beta_H(h_2)$, the left-hand side becomes 
$\beta_H(h_2)\lambda\bigl(\alpha_H(h_1)*\beta_H(g)\bigr)$.

By the definition of the untwisted product,
\[
\alpha_H(h_1)*\beta_H(g)
=
\alpha_H^{-1}(\alpha_H(h_1))\,
\beta_H^{-1}(\beta_H(g))
=
h_1g.
\]
Thus the left-hand side is $\beta_H(h_2)\lambda(h_1g)$.

On the right-hand side, since $\Delta_H(\beta_H(g))=\beta_H(g_1)\otimes\beta_H(g_2)$, 
we have
\[
\Delta^\sharp(\beta_H(g))
=
\alpha_H\beta_H(g_1)\otimes\beta_H^2(g_2).
\]
Therefore
\[
h*(\alpha_H\beta_H(g_1))
=
\alpha_H^{-1}(h)\,
\beta_H^{-1}(\alpha_H\beta_H(g_1))
=
\alpha_H^{-1}(h)\alpha_H(g_1),
\]
and
\[
S_H(\beta_H^2(g_2))
=
\beta_H^2(S_H(g_2)).
\]
Substituting these expressions into \eqref{eq:ordinary-haar-used}, we
obtain
\[
\beta_H(h_2)\lambda(h_1g)
=
\lambda\bigl(\alpha_H^{-1}(h)\alpha_H(g_1)\bigr)
\beta_H^2(S_H(g_2)).
\]
This proves \eqref{eq:transported-haar}.
\end{proof}

\begin{definition}\label{def:pi-and-kappa}
Let $(M,\alpha_M,\beta_M,\cdot,\rho_M)\in\mathcal{M}_A^H$. Define a linear map $\pi_M:M\longrightarrow M$ by
\begin{equation}\label{eq:pi-definition}
\pi_M(m)=\lambda(m_{[1]})\,\alpha_M(m_{[0]}),
\qquad m\in M.
\end{equation}
We also define
\begin{equation}\label{eq:kappa-definition}
\kappa(M) = \{\,m\in M\mid \pi_M(m\cdot a)=0,\ \forall a\in A\,\}.
\end{equation}
\end{definition}

\begin{lemma}\label{lem:pi-basic-properties}
Let $M\in\mathcal{M}_A^H$. Then:
\begin{enumerate}
\item $\operatorname{Im}\pi_M\subseteq M^{\operatorname{co}H}$;
\item $\pi_M|_{M^{\operatorname{co}H}}=\operatorname{id}_{M^{\operatorname{co}H}}$;
\item $\pi_M^2=\pi_M$ and $\operatorname{Im}\pi_M=M^{\operatorname{co}H}$;
\item $\pi_M\circ\alpha_M=\alpha_M\circ\pi_M$ and
      $\pi_M\circ\beta_M=\beta_M\circ\pi_M$;
\item $\pi_M$ is right $B$-linear;
\item for every morphism $f:M\to N$ in $\mathcal{M}_A^H$, one has $f\circ\pi_M=\pi_N\circ f$.
\end{enumerate}
\end{lemma}

\begin{proof}
We first show that $\pi_M(m)$ is coinvariant. Using the right
$H$-BiHom-comodule identity of $M$, we have
\begin{align*}
\rho_M(\pi_M(m))
&=
\lambda(m_{[1]})\,
\rho_M(\alpha_M(m_{[0]}))                                      \\
&=
\lambda(m_{[1]})\,
\alpha_M(m_{[0][0]})\otimes \alpha_H(m_{[0][1]})                \\
&=
m_{[0]}\otimes \alpha_H(m_{[1]1})\lambda(m_{[1]2})              \\
&=
\lambda(m_{[1]})m_{[0]}\otimes1_H                               \\
&=
\alpha_M^{-1}(\pi_M(m))\otimes1_H.
\end{align*}
Thus $\pi_M(m)\in M^{\operatorname{co}H}$.

If $n\in M^{\operatorname{co}H}$, then $\rho_M(n)=\alpha_M^{-1}(n)\otimes1_H$. Hence $\pi_M(n)=\lambda(1_H)\alpha_M(\alpha_M^{-1}(n))=n$. Therefore $\pi_M$ restricts to the identity on $M^{\operatorname{co}H}$.
It follows immediately that $\pi_M^2=\pi_M$ and
$\operatorname{Im}\pi_M=M^{\operatorname{co}H}$.

We also record that $\pi_M$ commutes with the two structure maps. Indeed,
using the colinearity of $\alpha_M$ and the invariance of $\lambda$, we get
\[
\pi_M(\alpha_M(m))
=
\lambda(\alpha_H(m_{[1]}))\,\alpha_M^2(m_{[0]})
=
\alpha_M(\pi_M(m)).
\]
Similarly,
\[
\pi_M(\beta_M(m))
=
\lambda(\beta_H(m_{[1]}))\,\alpha_M\beta_M(m_{[0]})
=
\beta_M(\pi_M(m)).
\]
Thus, $\pi_M\circ\alpha_M=\alpha_M\circ\pi_M$ and $\pi_M\circ\beta_M=\beta_M\circ\pi_M$.

Let $b\in B$. Since $\rho_A(b)=\alpha_A^{-1}(b)\otimes1_H$, we obtain
\begin{align*}
\pi_M(m\cdot b)
&=
\lambda(m_{[1]}1_H)\,
\alpha_M(m_{[0]}\cdot\alpha_A^{-1}(b))                         \\
&=
\lambda(\alpha_H(m_{[1]}))\,
\alpha_M(m_{[0]}\cdot\alpha_A^{-1}(b))                         \\
&=
\lambda(m_{[1]})\,\alpha_M(m_{[0]})\cdot b                      \\
&=
\pi_M(m)\cdot b.
\end{align*}
Thus $\pi_M$ is right $B$-linear.

Finally, if $f:M\to N$ is a morphism in $\mathcal{M}_A^H$, then $\rho_N(f(m))=f(m_{[0]})\otimes m_{[1]}$, and $f\alpha_M=\alpha_Nf$. Therefore
\[
\pi_N(f(m))
=
\lambda(m_{[1]})\alpha_N(f(m_{[0]}))
=
f(\lambda(m_{[1]})\alpha_M(m_{[0]}))
=
f(\pi_M(m)).
\]
\end{proof}

\begin{corollary}\label{cor:kappa-natural}
The assignment $M\mapsto\kappa(M)$ is a natural subfunctor of the
identity functor on $\mathcal M_A^H$.
\end{corollary}

\begin{proof}
If $f:M\to N$, $x\in\kappa(M)$, and $a\in A$, then $A$-linearity of
$f$ and Lemma~\ref{lem:pi-basic-properties} give
\[
\pi_N(f(x)\cdot a)
=\pi_N(f(x\cdot a))
=f(\pi_M(x\cdot a))=0.
\]
Therefore $f(\kappa(M))\subseteq\kappa(N)$, proving the claim.
\end{proof}

\begin{lemma}\label{lem:kappa-subobject}
For every $M\in\mathcal{M}_A^H$, the subspace $\kappa(M)$ is a
relative BiHom-Hopf submodule of $M$. Hence $\kappa(M)\in\mathcal{M}_A^H$.
\end{lemma}

\begin{proof}
First we show that $\kappa(M)$ is stable under the structure maps and
their inverses. Since $\pi_M$ commutes with $\alpha_M$ and $\beta_M$,
it also commutes with $\alpha_M^{-1}$ and $\beta_M^{-1}$.

For $m\in\kappa(M)$ and $a\in A$, we have
\[
\pi_M(\alpha_M(m)\cdot a)
=
\pi_M(\alpha_M(m\cdot\alpha_A^{-1}(a)))
=
\alpha_M(\pi_M(m\cdot\alpha_A^{-1}(a)))=0,
\]
and similarly $\pi_M(\beta_M(m)\cdot a)=0$. 
Thus
\[
\alpha_M(\kappa(M))\subseteq\kappa(M),
\qquad
\beta_M(\kappa(M))\subseteq\kappa(M).
\]

For the inverse maps, again for $m\in\kappa(M)$ and $a\in A$,
\[
\pi_M(\alpha_M^{-1}(m)\cdot a)
=
\pi_M(\alpha_M^{-1}(m\cdot\alpha_A(a)))
=
\alpha_M^{-1}(\pi_M(m\cdot\alpha_A(a)))=0,
\]
and similarly
\[
\pi_M(\beta_M^{-1}(m)\cdot a)
=
\pi_M(\beta_M^{-1}(m\cdot\beta_A(a)))
=
\beta_M^{-1}(\pi_M(m\cdot\beta_A(a)))=0.
\]
Hence
\[
\alpha_M^{-1}(\kappa(M))\subseteq\kappa(M),
\qquad
\beta_M^{-1}(\kappa(M))\subseteq\kappa(M).
\]

Next, $\kappa(M)$ is stable under the right $A$-action. Let
$m\in\kappa(M)$ and $c\in A$. For any $a\in A$, $(m\cdot c)\cdot a = \alpha_M(m)\cdot(c\beta_A^{-1}(a))$.

Therefore
\[
\pi_M((m\cdot c)\cdot a)
=
\pi_M(\alpha_M(m)\cdot(c\beta_A^{-1}(a)))=0.
\]
Thus $m\cdot c\in\kappa(M)$.

It remains to prove stability under the right $H$-coaction. We first
prove the following identity: for all $m\in M$ and $a\in A$,
\begin{equation}\label{eq:key-kappa-coaction-identity}
\pi_M(m_{[0]}\cdot a)\otimes m_{[1]}
=
\alpha_M^{-1}
\bigl(\pi_M(m\cdot\alpha_A^2(a_{[0]}))\bigr)
\otimes
\beta_H(S_H(a_{[1]})).
\end{equation}
Indeed, using the relative compatibility condition for $M$, we have
\begin{align*}
\pi_M(m_{[0]}\cdot a)\otimes m_{[1]}
&=
\lambda\bigl((m_{[0]}\cdot a)_{[1]}\bigr)
\alpha_M\bigl((m_{[0]}\cdot a)_{[0]}\bigr)
\otimes m_{[1]}                                                \\
&=
\lambda(m_{[0][1]}a_{[1]})
\alpha_M(m_{[0][0]}\cdot a_{[0]})
\otimes m_{[1]}                                                \\
&=
\bigl(\alpha_M(m_{[0][0]})\cdot\alpha_A(a_{[0]})\bigr)
\lambda(m_{[0][1]}a_{[1]})
\otimes m_{[1]}.
\end{align*}
Now apply the right $H$-BiHom-comodule identity of $M$,
\[
(\alpha_M^{-1}\otimes\Delta_H)\rho_M
=
(\rho_M\otimes\beta_H^{-1})\rho_M.
\]
Equivalently,
\[
\alpha_M(m_{[0][0]})\otimes m_{[0][1]}\otimes m_{[1]}
=
m_{[0]}\otimes m_{[1]1}\otimes\beta_H(m_{[1]2}).
\]
Thus
\begin{align*}
\pi_M(m_{[0]}\cdot a)\otimes m_{[1]}
&=
\bigl(m_{[0]}\cdot\alpha_A(a_{[0]})\bigr)
\lambda(m_{[1]1}a_{[1]})
\otimes\beta_H(m_{[1]2}).
\end{align*}
We now use the transported Haar identity
\eqref{eq:transported-haar} with $h=m_{[1]}$ and $g=a_{[1]}$.
It gives
\[
\beta_H(m_{[1]2})\lambda(m_{[1]1}a_{[1]})
=
\lambda\bigl(\alpha_H^{-1}(m_{[1]})\alpha_H(a_{[1]1})\bigr)
\beta_H^2(S_H(a_{[1]2})).
\]
Substituting this into the preceding expression yields
\begin{align*}
\pi_M(m_{[0]}\cdot a)\otimes m_{[1]}
&=
\bigl(m_{[0]}\cdot\alpha_A(a_{[0]})\bigr)
\lambda\bigl(\alpha_H^{-1}(m_{[1]})\alpha_H(a_{[1]1})\bigr)
\otimes
\beta_H^2(S_H(a_{[1]2})).
\end{align*}

Next we rewrite the $A$-coaction terms. Since $A$ is a right
$H$-BiHom-comodule, we have
\[
(\alpha_A^{-1}\otimes\Delta_H)\rho_A
=
(\rho_A\otimes\beta_H^{-1})\rho_A.
\]
Equivalently,
\[
\alpha_A^{-1}(a_{[0]})\otimes a_{[1]1}\otimes a_{[1]2}
=
a_{[0][0]}\otimes a_{[0][1]}\otimes\beta_H^{-1}(a_{[1]}).
\]
Applying $\alpha_A^2\otimes\alpha_H^2\otimes \beta_H^2S_H$ to the three tensor factors gives
\[
\alpha_A(a_{[0]})\otimes\alpha_H^2(a_{[1]1})
\otimes\beta_H^2(S_H(a_{[1]2}))
=
\alpha_A^2(a_{[0][0]})\otimes\alpha_H^2(a_{[0][1]})
\otimes\beta_H(S_H(a_{[1]})).
\]
Furthermore, using the multiplicativity of $\alpha_H$ and the identity
$\lambda\circ\alpha_H=\lambda$, we have
\[
\lambda\bigl(\alpha_H^{-1}(m_{[1]})\alpha_H(a_{[1]1})\bigr)
=
\lambda\bigl(m_{[1]}\alpha_H^2(a_{[1]1})\bigr).
\]
Hence
\[
\pi_M(m_{[0]}\cdot a)\otimes m_{[1]}
=
\bigl(m_{[0]}\cdot\alpha_A^2(a_{[0][0]})\bigr)
\lambda\bigl(m_{[1]}\alpha_H^2(a_{[0][1]})\bigr)
\otimes\beta_H(S_H(a_{[1]})).
\]

Finally, compute $\alpha_M^{-1} \bigl(\pi_M(m\cdot\alpha_A^2(a_{[0]}))\bigr)$. By the relative compatibility condition,
\[
\rho_M(m\cdot\alpha_A^2(a_{[0]}))
=
m_{[0]}\cdot\alpha_A^2(a_{[0][0]})
\otimes
m_{[1]}\alpha_H^2(a_{[0][1]}).
\]
Therefore
\begin{align*}
\alpha_M^{-1}
\bigl(\pi_M(m\cdot\alpha_A^2(a_{[0]}))\bigr)
&=
\alpha_M^{-1}
\left[
\lambda\bigl(m_{[1]}\alpha_H^2(a_{[0][1]})\bigr)
\alpha_M\bigl(m_{[0]}\cdot\alpha_A^2(a_{[0][0]})\bigr)
\right]                                                       \\
&=
\bigl(m_{[0]}\cdot\alpha_A^2(a_{[0][0]})\bigr)
\lambda\bigl(m_{[1]}\alpha_H^2(a_{[0][1]})\bigr).
\end{align*}
Combining the last two displayed formulas proves
\eqref{eq:key-kappa-coaction-identity}.

Now let $m\in\kappa(M)$. Then the right-hand side of
\eqref{eq:key-kappa-coaction-identity} is zero for every $a\in A$.
Hence
\[
\pi_M(m_{[0]}\cdot a)\otimes m_{[1]}=0,
\qquad a\in A.
\]
Write $\rho_M(m)=\sum_i m_i\otimes h_i$ with the elements $h_i$ linearly independent. Then
\[
\sum_i \pi_M(m_i\cdot a)\otimes h_i=0,
\qquad a\in A.
\]
By the linear independence of the $h_i$, we get
\[
\pi_M(m_i\cdot a)=0,
\qquad \forall a\in A,\ \forall i.
\]
Thus each $m_i$ lies in $\kappa(M)$, so
$\rho_M(m)=\sum_i m_i\otimes h_i\in\kappa(M)\otimes H$.
Hence $\kappa(M)$ is stable under the right $H$-coaction.
\end{proof}

\begin{lemma}\label{lem:kappa-quotient}
For every $M\in\mathcal{M}_A^H$, one has $\kappa(M/\kappa(M))=0$.
\end{lemma}

\begin{proof}
First observe that $\pi_M(\kappa(M))=0$. Indeed, if $m\in\kappa(M)$, then, taking $a=1_A$ in the definition
of $\kappa(M)$, we get $0=\pi_M(m\cdot1_A)=\pi_M(\alpha_M(m))$. 
Since $\pi_M$ commutes with $\alpha_M$, this gives $0=\alpha_M(\pi_M(m))$. As $\alpha_M$ is bijective, we obtain $\pi_M(m)=0$.

We also need the elementary fact
\begin{equation}\label{eq:kappa-intersection-zero}
\kappa(M)\cap M^{\operatorname{co}H}=0.
\end{equation}
Indeed, if $x\in\kappa(M)\cap M^{\operatorname{co}H}$, then
$x=\pi_M(x)$. Since $x\in\kappa(M)$, taking $a=1_A$ gives
\[
0=\pi_M(x\cdot1_A)=\pi_M(\alpha_M(x)).
\]
Using $\pi_M\alpha_M=\alpha_M\pi_M$, we get $0=\alpha_M(\pi_M(x))=\alpha_M(x)$. Since $\alpha_M$ is bijective, $x=0$.

Let $\overline m=m+\kappa(M)\in\kappa(M/\kappa(M))$. For every $a\in A$,
\[
0
=
\pi_{M/\kappa(M)}(\overline m\cdot a)
=
\pi_M(m\cdot a)+\kappa(M).
\]
Thus $\pi_M(m\cdot a)\in\kappa(M)$. But
$\pi_M(m\cdot a)\in M^{\operatorname{co}H}$ by
Lemma~\ref{lem:pi-basic-properties}. Hence
\[
\pi_M(m\cdot a)\in \kappa(M)\cap M^{\operatorname{co}H}=0.
\]
Therefore $\pi_M(m\cdot a)=0$ for all $a\in A$, so $m\in\kappa(M)$.
Thus $\overline m=0$, and consequently $\kappa(M/\kappa(M))=0$.
\end{proof}

\begin{lemma}\label{lem:coinvariant-exact}
The coinvariant functor $G$ is exact. Consequently, for every $M\in\mathcal{M}_A^H$, $(M/\kappa(M))^{\operatorname{co}H}\cong M^{\operatorname{co}H}$ as right $B$-BiHom-modules.
\end{lemma}

\begin{proof}
Since $G$ is right adjoint to the induction functor by
Theorem~\ref{thm:induction-coinvariant-adjunction}, it is left exact.

It remains to prove right exactness. Let $f:M\longrightarrow N$ be an epimorphism in $\mathcal{M}_A^H$, and let
$n\in N^{\operatorname{co}H}$. Choose $m\in M$ such that $f(m)=n$.
Since $f$ is right $H$-colinear,
\[
f(m_{[0]})\otimes m_{[1]}
=
\alpha_N^{-1}(n)\otimes1_H.
\]
Applying the map $\pi$ and using the naturality proved in
Lemma~\ref{lem:pi-basic-properties}, we obtain
\[
f(\pi_M(m))
=
\pi_N(f(m))
=
\pi_N(n)
=
n.
\]
Since $\pi_M(m)\in M^{\operatorname{co}H}$, the induced map
\[
f|_{M^{\operatorname{co}H}}:
M^{\operatorname{co}H}\longrightarrow N^{\operatorname{co}H}
\]
is surjective. Hence $G$ is exact.

Now apply $G$ to the short exact sequence
\[
0\longrightarrow \kappa(M)\longrightarrow M
\longrightarrow M/\kappa(M)\longrightarrow0.
\]
By \eqref{eq:kappa-intersection-zero}, one has $\kappa(M)^{\operatorname{co}H}=0$. Therefore
\[
(M/\kappa(M))^{\operatorname{co}H}
\cong
M^{\operatorname{co}H}/\kappa(M)^{\operatorname{co}H}
=
M^{\operatorname{co}H}.
\]
This isomorphism is induced by the quotient map
$q_M:M\to M/\kappa(M)$ and is natural in $M$.
\end{proof}

\begin{theorem}\label{thm:torsion-theory}
Let
\[
\mathcal{U}
=
\{\,T\in\mathcal{M}_A^H\mid \kappa(T)=T\,\},
\qquad
\mathcal{V}
=
\{\,F\in\mathcal{M}_A^H\mid \kappa(F)=0\,\}.
\]
Then $(\mathcal{U},\mathcal{V})$ is a hereditary torsion theory in
$\mathcal{M}_A^H$.
\end{theorem}

\begin{proof}
For every $M\in\mathcal{M}_A^H$, Lemma~\ref{lem:kappa-subobject}
gives a short exact sequence
\[
0\longrightarrow \kappa(M)\longrightarrow M
\longrightarrow M/\kappa(M)\longrightarrow0.
\]
For $x\in\kappa(M)$, the projection and action on the subobject
$\kappa(M)$ are the restrictions of those on $M$.  Hence
$\pi_{\kappa(M)}(x\cdot a)=\pi_M(x\cdot a)=0$ for all $a\in A$, so
$\kappa(\kappa(M))=\kappa(M)$ and $\kappa(M)\in\mathcal{U}$.
By Lemma~\ref{lem:kappa-quotient}, $\kappa(M/\kappa(M))=0$, so $M/\kappa(M)\in\mathcal{V}$.

Let $T\in\mathcal{U}$, $F\in\mathcal{V}$, and let
$f:T\to F$ be a morphism in $\mathcal{M}_A^H$. By the naturality of $\pi$, we have $f(\kappa(T))\subseteq\kappa(F)$. Since $\kappa(T)=T$ and $\kappa(F)=0$, we get $f=0$.

Finally, let $T\in\mathcal{U}$ and let $i:X\hookrightarrow T$ be a subobject in $\mathcal{M}_A^H$. If $x\in X$, then
$i(x)\in T=\kappa(T)$, hence for every $a\in A$, $\pi_T(i(x)\cdot a)=0$. By naturality of $\pi$, we have
\[
i\bigl(\pi_X(x\cdot a)\bigr)
=
\pi_T(i(x)\cdot a)
=
0.
\]
Since $i$ is monic, it follows that
\[
\pi_X(x\cdot a)=0,
\qquad a\in A.
\]
Thus $x\in\kappa(X)$, so $X=\kappa(X)$ and therefore $X\in\mathcal{U}$.
Hence the torsion theory is hereditary.

\end{proof}

\begin{remark}\label{rem:torsion-kernel-coinvariants}
The torsion class has the description $\mathcal U=\ker G$.  Indeed, if
$\kappa(T)=T$, then
$T^{\operatorname{co}H}=T^{\operatorname{co}H}\cap\kappa(T)=0$ by
\eqref{eq:kappa-intersection-zero}.  Conversely, if
$T^{\operatorname{co}H}=0$, then
$\operatorname{Im}\pi_T=T^{\operatorname{co}H}=0$, and therefore
$\kappa(T)=T$.
\end{remark}

\section{Categorical equivalence via modified induction}

Set $Q(M)=\overline M:=M/\kappa(M)$ and write $\overline m$ for the
image of $m\in M$. By Lemma~\ref{lem:kappa-quotient}, $\overline M$ is
torsion-free.

The \textit{modified induction functor}
\[
\overline F=\overline{-\otimes_BA}:
\mathcal{M}_B\longrightarrow\mathcal{M}_A^H
\]
is defined by
$\overline F(N)
=
\overline{N\otimes_BA}
=
(N\otimes_BA)/\kappa(N\otimes_BA)$.

\begin{proposition}\label{prop:modified-induction-functor}
The above assignment extends to a well-defined functor $\overline F$.
For a morphism $f:N\to N'$ in $\mathcal{M}_B$, the morphism
\[
F(f):N\otimes_BA\longrightarrow N'\otimes_BA
\quad
\text{induces}\quad
\overline F(f):
\overline{N\otimes_BA}\longrightarrow
\overline{N'\otimes_BA}.
\]
\end{proposition}

\begin{proof}
It remains only to check that $F(f)(\kappa(N\otimes_BA))\subseteq\kappa(N'\otimes_BA)$.
Let $x\in\kappa(N\otimes_BA)$. For every $a\in A$, by the naturality of
$\pi$ we have
\[
\pi_{N'\otimes_BA}(F(f)(x)\cdot a)
=
F(f)\bigl(\pi_{N\otimes_BA}(x\cdot a)\bigr)
=
0.
\]
Therefore $F(f)(x)\in\kappa(N'\otimes_BA)$. Hence $F(f)$ descends to the
quotients, and the induced maps preserve identities and composition.
\end{proof}

\begin{proposition}\label{prop:factor-adjunction}
For every $N\in\mathcal{M}_B$ and every $M\in\mathcal{M}_A^H$, there is
a natural bijection
\[
\operatorname{Hom}_{\mathcal{M}_A^H}
(\overline F(N),\overline M)
\cong
\operatorname{Hom}_{\mathcal{M}_B}
(N,M^{\operatorname{co}H}).
\]
In particular, if $M$ is torsion-free, then
\[
\operatorname{Hom}_{\mathcal{M}_A^H}
(\overline F(N),M)
\cong
\operatorname{Hom}_{\mathcal{M}_B}
(N,M^{\operatorname{co}H}).
\]
\end{proposition}

\begin{proof}
By Lemma~\ref{lem:coinvariant-exact}, we have a natural isomorphism $\overline M^{\operatorname{co}H}\cong M^{\operatorname{co}H}$. Thus, it suffices to construct a natural bijection
\[
\operatorname{Hom}_{\mathcal{M}_A^H}
(\overline F(N),\overline M)
\cong
\operatorname{Hom}_{\mathcal{M}_B}
(N,\overline M^{\operatorname{co}H}).
\]

Let $f:\overline F(N)\longrightarrow \overline M$ be a morphism in $\mathcal{M}_A^H$. Define $\chi(f):N\longrightarrow \overline M^{\operatorname{co}H}$ by
\[
\chi(f)(n)=f\bigl(\overline{\alpha_N^{-1}(n)\otimes_B1_A}\bigr).
\]
The element $\overline{\alpha_N^{-1}(n)\otimes_B1_A}$ is coinvariant,
so the $H$-colinearity of $f$ implies that $\chi(f)(n)$ is coinvariant.

We check that $\chi(f)$ is a morphism in $\mathcal{M}_B$. For the
structure maps,
\begin{align*}
\chi(f)(\alpha_N(n))
&=
f\bigl(\overline{n\otimes_B1_A}\bigr) \\
&=
\alpha_{\overline M}
\Bigl(f\bigl(\overline{\alpha_N^{-1}(n)\otimes_B1_A}\bigr)\Bigr) \\
&=
\alpha_{\overline M}(\chi(f)(n)),
\end{align*}
and similarly $\chi(f)(\beta_N(n))=\beta_{\overline M}(\chi(f)(n))$. 
For $b\in B$,
\[
\overline{\alpha_N^{-1}(n\cdot b)\otimes_B1_A}
=
\overline{n\otimes_B b}
=
\bigl(\overline{\alpha_N^{-1}(n)\otimes_B1_A}\bigr)\cdot b,
\]
hence
\[
\chi(f)(n\cdot b)=\chi(f)(n)\cdot b.
\]
So
\[
\chi(f)\in \operatorname{Hom}_{\mathcal{M}_B}
(N,\overline M^{\operatorname{co}H}).
\]

Conversely, let $g:N\longrightarrow \overline M^{\operatorname{co}H}$ be a morphism in $\mathcal{M}_B$. First define
\[
\widetilde{\omega}(g):N\otimes_BA\longrightarrow \overline M
\]
by $\widetilde{\omega}(g)(n\otimes_Ba)=g(n)\cdot a$.

We first check that $\widetilde{\omega}(g)$ is well-defined. For
$b\in B$,
\begin{align*}
\widetilde{\omega}(g)\bigl((n\cdot b)\otimes_B\beta_A(a)\bigr)
&=
g(n\cdot b)\cdot\beta_A(a) \\
&=
(g(n)\cdot b)\cdot\beta_A(a) \\
&=
\alpha_{\overline M}(g(n))\cdot(ba),
\end{align*}
whereas
\begin{align*}
\widetilde{\omega}(g)\bigl(\alpha_N(n)\otimes_Bba\bigr)
&=
g(\alpha_N(n))\cdot(ba) \\
&=
\alpha_{\overline M}(g(n))\cdot(ba).
\end{align*}
Thus $\widetilde{\omega}(g)$ respects the balanced relation.

Next, $\widetilde{\omega}(g)$ commutes with the structure maps:
\begin{align*}
\widetilde{\omega}(g)\bigl(\alpha_N(n)\otimes_B\alpha_A(a)\bigr)
&=
g(\alpha_N(n))\cdot\alpha_A(a) \\
&=
\alpha_{\overline M}(g(n))\cdot\alpha_A(a) \\
&=
\alpha_{\overline M}(g(n)\cdot a)
=
\alpha_{\overline M}\bigl(\widetilde{\omega}(g)(n\otimes_Ba)\bigr),
\end{align*}
and similarly
\[
\widetilde{\omega}(g)\bigl(\beta_N(n)\otimes_B\beta_A(a)\bigr)
=
\beta_{\overline M}\bigl(\widetilde{\omega}(g)(n\otimes_Ba)\bigr).
\]

It is right $A$-linear because
\begin{align*}
\widetilde{\omega}(g)\bigl((n\otimes_Ba)\cdot c\bigr)
&=
\widetilde{\omega}(g)\bigl(\alpha_N(n)\otimes_Ba\beta_A^{-1}(c)\bigr) \\
&=
g(\alpha_N(n))\cdot(a\beta_A^{-1}(c)) \\
&=
\alpha_{\overline M}(g(n))\cdot(a\beta_A^{-1}(c)) \\
&=
(g(n)\cdot a)\cdot c \\
&=
\widetilde{\omega}(g)(n\otimes_Ba)\cdot c.
\end{align*}

It is right $H$-colinear since $g(n)$ is coinvariant:
\begin{align*}
\rho_{\overline M}(\widetilde{\omega}(g)(n\otimes_Ba))
&=
\rho_{\overline M}(g(n)\cdot a) \\
&=
g(n)_{[0]}\cdot a_{[0]}\otimes g(n)_{[1]}a_{[1]} \\
&=
\alpha_{\overline M}^{-1}(g(n))\cdot a_{[0]}
\otimes 1_Ha_{[1]} \\
&=
\alpha_{\overline M}^{-1}(g(n))\cdot a_{[0]}
\otimes \beta_H(a_{[1]}) \\
&=
(\widetilde{\omega}(g)\otimes\operatorname{id}_H)
\rho_{F(N)}(n\otimes_Ba).
\end{align*}

Now let $x\in\kappa(N\otimes_BA)$. For every $c\in A$, by naturality of
$\pi$,
\[
\pi_{\overline M}\bigl(\widetilde{\omega}(g)(x)\cdot c\bigr)
=
\widetilde{\omega}(g)\bigl(\pi_{N\otimes_BA}(x\cdot c)\bigr)
=
0.
\]
Hence $\widetilde{\omega}(g)(x)\in\kappa(\overline M)=0$, because $\overline M$ is torsion-free. 
Therefore $\widetilde{\omega}(g)$ descends uniquely to a morphism $\omega(g):\overline F(N)\longrightarrow \overline M$ given by $\omega(g)(\overline{n\otimes_Ba})=g(n)\cdot a$.

We now verify that $\chi$ and $\omega$ are inverse to each other.
For $g:N\to \overline M^{\operatorname{co}H}$, we have
\begin{align*}
(\chi\omega)(g)(n)
&=
\omega(g)\bigl(\overline{\alpha_N^{-1}(n)\otimes_B1_A}\bigr) \\
&=
g(\alpha_N^{-1}(n))\cdot1_A \\
&=
\alpha_{\overline M}^{-1}(g(n))\cdot1_A \\
&=
g(n).
\end{align*}

For $f:\overline F(N)\to \overline M$, we have
\begin{align*}
(\omega\chi)(f)(\overline{n\otimes_Ba})
&=
\chi(f)(n)\cdot a \\
&=
f\bigl(\overline{\alpha_N^{-1}(n)\otimes_B1_A}\bigr)\cdot a \\
&=
f\Bigl(\bigl(\overline{\alpha_N^{-1}(n)\otimes_B1_A}\bigr)\cdot a\Bigr) \\
&=
f(\overline{n\otimes_Ba}).
\end{align*}
Therefore $\chi$ and $\omega$ are inverse bijections.

The construction is natural in both $N$ and $M$.
\end{proof}

\begin{proposition}\label{prop:torsion-free-reflector}
Let $\mathcal V\subseteq\mathcal M_A^H$ be the full subcategory of
torsion-free objects.  The assignment
\[
Q(M)=M/\kappa(M)
\]
extends to a functor $Q:\mathcal M_A^H\to\mathcal V$ that is left
adjoint to the inclusion $\mathcal V\hookrightarrow\mathcal M_A^H$.
Consequently, $\overline F=QF$, and for $V\in\mathcal V$ there are
natural bijections
\[
\operatorname{Hom}_{\mathcal V}(QF(N),V)
\cong
\operatorname{Hom}_{\mathcal M_A^H}(F(N),V)
\cong
\operatorname{Hom}_{\mathcal M_B}(N,G(V)).
\]
\end{proposition}

\begin{proof}
Naturality of $\pi$ implies that every morphism $f:M\to M'$ satisfies
$f(\kappa(M))\subseteq\kappa(M')$; hence $Q$ is a functor.  If $V$ is
torsion-free and $f:M\to V$, then
$f(\kappa(M))\subseteq\kappa(V)=0$.  Thus $f$ factors uniquely through
the quotient $q_M:M\to Q(M)$.  This factorization is natural in both
variables and proves $Q\dashv(\mathcal V\hookrightarrow\mathcal M_A^H)$.
The displayed bijections now follow from this reflection and from
$F\dashv G$.
\end{proof}

An object $M\in\mathcal{M}_A^H$ is called \textit{$0$-generated} if the
canonical morphism
\[
\delta_M:M^{\operatorname{co}H}\otimes_BA\longrightarrow M,
\qquad
\delta_M(m\otimes_Ba)=m\cdot a,
\]
is an epimorphism. Equivalently,
\[
M=M^{\operatorname{co}H}\cdot A.
\]
The category $\mathcal{M}_A^H$ is called $0$-generated if every object
of $\mathcal{M}_A^H$ is $0$-generated.

\begin{lemma}\label{lem:ordinary-unit-isomorphism}
For every $N\in\mathcal{M}_B$, the unit
\[
\eta_N:
N\longrightarrow (N\otimes_BA)^{\operatorname{co}H},
\qquad
\eta_N(n)=\alpha_N^{-1}(n)\otimes_B1_A
\]
is an isomorphism in $\mathcal{M}_B$.
\end{lemma}

\begin{proof}
View $A$ as an object of $\mathcal{M}_A^H$ via its regular right
$A$-action. By Lemma~\ref{lem:pi-basic-properties} applied to $A$, the
map
\[
\pi_A(a)=\lambda(a_{[1]})\alpha_A(a_{[0]})
\]
takes values in $B=A^{\operatorname{co}H}$ and commutes with
$\beta_A$.

For $b\in B$ and $a\in A$, we also have
\[
\pi_A(ba)
=
\lambda((ba)_{[1]})\,\alpha_A((ba)_{[0]})
=
\lambda(1_Ha_{[1]})\,\alpha_A(\alpha_A^{-1}(b)a_{[0]})
=
b\,\lambda(a_{[1]})\alpha_A(a_{[0]})
=
b\,\pi_A(a),
\]
since
\[
\rho_A(b)=\alpha_A^{-1}(b)\otimes1_H,
\qquad
1_Ha_{[1]}=\beta_H(a_{[1]}),
\qquad
\lambda\circ\beta_H=\lambda.
\]

Define
\[
\widetilde{\xi}_N:N\otimes_BA\longrightarrow N,
\qquad
\widetilde{\xi}_N(n\otimes_Ba)=n\cdot\pi_A(a).
\]
We check that $\widetilde{\xi}_N$ is well-defined. For $b\in B$,
\begin{align*}
\widetilde{\xi}_N\bigl((n\cdot b)\otimes_B\beta_A(a)\bigr)
&=
(n\cdot b)\cdot\pi_A(\beta_A(a)) \\
&=
(n\cdot b)\cdot\beta_B(\pi_A(a)) \\
&=
\alpha_N(n)\cdot\bigl(b\,\pi_A(a)\bigr) \\
&=
\alpha_N(n)\cdot\pi_A(ba) \\
&=
\widetilde{\xi}_N\bigl(\alpha_N(n)\otimes_Bba\bigr).
\end{align*}
Thus $\widetilde{\xi}_N$ is well-defined.

Restrict $\widetilde{\xi}_N$ to the coinvariants:
\[
\xi_N:=\widetilde{\xi}_N|_{(N\otimes_BA)^{\operatorname{co}H}}
:
(N\otimes_BA)^{\operatorname{co}H}\longrightarrow N.
\]
For $n\in N$,
\[
(\xi_N\eta_N)(n)
=
\xi_N(\alpha_N^{-1}(n)\otimes_B1_A)
=
\alpha_N^{-1}(n)\cdot\pi_A(1_A)
=
\alpha_N^{-1}(n)\cdot1_B
=
n.
\]

For a pure tensor $n\otimes_Ba$, we compute
\[
\pi_{N\otimes_BA}(n\otimes_Ba)
=
\lambda(\beta_H(a_{[1]}))\,
\alpha_{F(N)}(\alpha_N^{-1}(n)\otimes_Ba_{[0]})
=
n\otimes_B\pi_A(a).
\]
The balancing relation gives
\begin{align*}
\eta_N\bigl(\widetilde\xi_N(n\otimes_Ba)\bigr)
&=\eta_N(n\cdot\pi_A(a))\\
&=\bigl(\alpha_N^{-1}(n)\cdot
\alpha_B^{-1}(\pi_A(a))\bigr)\otimes_B1_A\\
&=n\otimes_B\pi_A(a).
\end{align*}
In the last equality we used the balancing relation and
$\alpha_B^{-1}(\pi_A(a))1_A=\pi_A(a)$.  Thus
$\eta_N\widetilde\xi_N=\pi_{N\otimes_BA}$.  By linearity, for every
$x\in (N\otimes_BA)^{\operatorname{co}H}$,
\[
(\eta_N\xi_N)(x)=\pi_{N\otimes_BA}(x)=x,
\]
because $\pi_{N\otimes_BA}$ is the identity on coinvariants. Therefore
$\eta_N$ is an isomorphism in $\mathcal M_B$ with inverse $\xi_N$.
\end{proof}

\begin{lemma}\label{lem:unique-extension}
Let $M\in\mathcal{M}_A^H$ be $0$-generated and let $N\in\mathcal{M}_B$. For every morphism $\theta:M^{\operatorname{co}H}\longrightarrow N$ in $\mathcal{M}_B$, there exists a unique morphism
\[
\vartheta:M\longrightarrow \overline{N\otimes_BA}
\]
in $\mathcal{M}_A^H$ such that
\[
\vartheta(m)=
\overline{\alpha_N^{-1}(\theta(m))\otimes_B1_A},
\qquad m\in M^{\operatorname{co}H}.
\]
\end{lemma}

\begin{proof}
Set $Y:=M^{\operatorname{co}H}\otimes_BA$. Since $M$ is $0$-generated, the counit $\delta_M:Y\longrightarrow M$ is an epimorphism. Define
\[
\phi:Y\longrightarrow \overline{N\otimes_BA},
\qquad
\phi(m\otimes_Ba)=\overline{\theta(m)\otimes_Ba}.
\]
This is a morphism in $\mathcal{M}_A^H$.

We claim that $\ker\delta_M\subseteq\kappa(Y)$. Let $y\in\ker\delta_M$ and $a\in A$. Since $\ker\delta_M$ is an
$A$-submodule of $Y$, we have $y\cdot a\in\ker\delta_M$. By naturality
of $\pi$,
\[
\delta_M\bigl(\pi_Y(y\cdot a)\bigr)
=
\pi_M\bigl(\delta_M(y\cdot a)\bigr)
=
0.
\]
Now $\pi_Y(y\cdot a)\in Y^{\operatorname{co}H}$, and $G(\delta_M):Y^{\operatorname{co}H}\longrightarrow M^{\operatorname{co}H}$ is an isomorphism: indeed, by the triangular identity,
\[
G(\delta_M)\circ\eta_{M^{\operatorname{co}H}}
=
\operatorname{id}_{M^{\operatorname{co}H}},
\]
and $\eta_{M^{\operatorname{co}H}}$ is an isomorphism by
Lemma~\ref{lem:ordinary-unit-isomorphism}. Hence
\[
\pi_Y(y\cdot a)=0,
\qquad a\in A,
\]
so $y\in\kappa(Y)$. This proves $\ker\delta_M\subseteq\kappa(Y)$.

Next, we show that $\kappa(Y)\subseteq\ker\phi$. Let $y\in\kappa(Y)$. Then for every $a\in A$, naturality of $\pi$ gives
\[
\pi_{\overline{N\otimes_BA}}(\phi(y)\cdot a)
=
\phi\bigl(\pi_Y(y\cdot a)\bigr)
=
0.
\]
Hence $\phi(y)\in\kappa(\overline{N\otimes_BA})=0$, because
$\overline{N\otimes_BA}$ is torsion-free. Therefore $\kappa(Y)\subseteq\ker\phi$.

We have proved
\[
\ker\delta_M\subseteq\kappa(Y)\subseteq\ker\phi.
\]
Hence $\phi$ factors uniquely through $\delta_M$, giving a unique
morphism
\[
\vartheta:M\longrightarrow\overline{N\otimes_BA}
\]
such that
\[
\vartheta\circ\delta_M=\phi.
\]

For $m\in M^{\operatorname{co}H}$, we have $m=\alpha_M^{-1}(m)\cdot1_A$, and therefore
\begin{align*}
\vartheta(m)
&=\vartheta\bigl(\delta_M(\alpha_M^{-1}(m)\otimes_B1_A)\bigr)\\
&=\phi(\alpha_M^{-1}(m)\otimes_B1_A)\\
&=\overline{\theta(\alpha_M^{-1}(m))\otimes_B1_A}\\
&=\overline{\alpha_N^{-1}(\theta(m))\otimes_B1_A}.
\end{align*}
Uniqueness follows because $\delta_M$ is an epimorphism.
\end{proof}

\begin{lemma}\label{lem:factor-unit-isomorphism}
For every $N\in\mathcal{M}_B$, the factor unit
\[
\overline\eta_N:
N\longrightarrow
(\overline{N\otimes_BA})^{\operatorname{co}H},
\qquad
\overline\eta_N(n)=\overline{\alpha_N^{-1}(n)\otimes_B1_A},
\]
is an isomorphism in $\mathcal{M}_B$.
\end{lemma}

\begin{proof}
Let
\[
q_N:N\otimes_BA\longrightarrow \overline{N\otimes_BA}
\]
be the canonical quotient map. Since the coinvariant functor $G$ is exact
by Lemma~\ref{lem:coinvariant-exact}, applying $G$ to
\[
0\longrightarrow\kappa(N\otimes_BA)
\longrightarrow N\otimes_BA\xrightarrow{q_N}
\overline{N\otimes_BA}\longrightarrow0
\]
gives a short exact sequence.  Moreover,
\[
G(\kappa(N\otimes_BA))
=\kappa(N\otimes_BA)\cap(N\otimes_BA)^{\operatorname{co}H}=0
\]
by \eqref{eq:kappa-intersection-zero}.  Hence the induced map
\[
G(q_N):
(N\otimes_BA)^{\operatorname{co}H}
\longrightarrow
(\overline{N\otimes_BA})^{\operatorname{co}H}
\]
is an isomorphism.

By definition of the factor unit, $\overline{\eta}_N=G(q_N)\circ \eta_N$. Now $\eta_N$ is an isomorphism by Lemma~\ref{lem:ordinary-unit-isomorphism}. Therefore $\overline{\eta}_N$ is an isomorphism.
\end{proof}

\begin{lemma}\label{lem:factor-counit-isomorphism}
Let $M\in\mathcal{M}_A^H$ be $0$-generated. Then the factor counit
\[
\overline\delta_M:
\overline{M^{\operatorname{co}H}\otimes_BA}
\longrightarrow
\overline M,
\qquad
\overline\delta_M(\overline{m\otimes_Ba})
=
\overline{m\cdot a},
\]
is an isomorphism in $\mathcal{M}_A^H$.
\end{lemma}

\begin{proof}
First we verify that the displayed map is well-defined.  If
$x\in\kappa(M^{\operatorname{co}H}\otimes_BA)$ and $a\in A$, then
$A$-linearity of $\delta_M$ and naturality of $\pi$ give
\[
\pi_M(\delta_M(x)\cdot a)
=\pi_M(\delta_M(x\cdot a))
=\delta_M\bigl(\pi_{F(G(M))}(x\cdot a)\bigr)=0.
\]
Thus $\delta_M(x)\in\kappa(M)$, so $\delta_M$ descends to
$\overline\delta_M$.

First assume that $M$ is torsion-free, so that $\overline M=M$, and set $Z:=\overline{M^{\operatorname{co}H}\otimes_BA}$. 
The object $Z$ is $0$-generated, because every element $\overline{m\otimes_Ba}\in Z$ can be written as
\[
\overline{m\otimes_Ba}
=
\bigl(\overline{\alpha_M^{-1}(m)\otimes_B1_A}\bigr)\cdot a,
\]
and the first factor is coinvariant.

By Lemma~\ref{lem:unique-extension}, applied to $\theta=\operatorname{id}_{M^{\operatorname{co}H}}$, there exists a morphism $\vartheta:M\longrightarrow Z$ such that
\[
\vartheta(m)=\overline{\alpha_M^{-1}(m)\otimes_B1_A},
\qquad m\in M^{\operatorname{co}H}.
\]

By Proposition~\ref{prop:factor-adjunction}, the factor counit $\overline\delta_M:Z\longrightarrow M$ corresponds to $\operatorname{id}_{M^{\operatorname{co}H}}$, hence
\[
\overline\delta_M\bigl(\overline{\alpha_M^{-1}(m)\otimes_B1_A}\bigr)=m,
\qquad m\in M^{\operatorname{co}H}.
\]
Therefore
\[
(\overline\delta_M\vartheta)(m)=m,
\qquad m\in M^{\operatorname{co}H}.
\]
Since $M$ is $0$-generated and both morphisms are right $A$-linear, it
follows that $\overline\delta_M\vartheta=\operatorname{id}_M$.

On the other hand, for $m\in M^{\operatorname{co}H}$,
\[
(\vartheta\overline\delta_M)
\bigl(\overline{\alpha_M^{-1}(m)\otimes_B1_A}\bigr)
=
\vartheta(m)
=
\overline{\alpha_M^{-1}(m)\otimes_B1_A}.
\]
Since $Z$ is $0$-generated and both morphisms are right $A$-linear, we
obtain $\vartheta\overline\delta_M=\operatorname{id}_Z$. 
Thus $\overline\delta_M$ is an isomorphism whenever $M$ is torsion-free.

Now let $M$ be an arbitrary $0$-generated object. Let $q:M\longrightarrow \overline M$ be the quotient map. Since $M$ is $0$-generated, the composite
\[
q\circ\delta_M:
M^{\operatorname{co}H}\otimes_BA\longrightarrow \overline M
\]
is an epimorphism. By Lemma~\ref{lem:coinvariant-exact}, the map
$G(q):M^{\operatorname{co}H}\to\overline M^{\operatorname{co}H}$ is
an isomorphism. Under this identification, the above epimorphism is
exactly the canonical counit for $\overline M$, since
\[
\delta_{\overline M}\bigl(G(q)(m)\otimes_Ba\bigr)
=q(m)\cdot a=q(m\cdot a).
\]
Hence $\overline M$ is again $0$-generated. Since $\overline M$ is
torsion-free by Lemma~\ref{lem:kappa-quotient}, the torsion-free case
applies to $\overline M$. Equivalently, on passing to torsion-free
factors the preceding identity gives
\[
\overline\delta_{\overline M}\,\overline F(G(q))
=\overline\delta_M.
\]
Both $G(q)$ and $\overline F(G(q))$ are isomorphisms, so the displayed
morphism
\[
\overline\delta_M:
\overline{M^{\operatorname{co}H}\otimes_BA}\longrightarrow \overline M
\]
is an isomorphism.
\end{proof}

\begin{theorem}\label{thm:main-equivalence}
Let $H$ be a monoidal BiHom-Hopf algebra with a fixed Haar integral
$\lambda\in H^*$, let $A$ be a right
$H$-BiHom-comodule algebra, and put $B=A^{\operatorname{co}H}$. Let
$(\mathcal{M}_A^H)_{\operatorname{tf},0}$ be the full subcategory of
$\mathcal{M}_A^H$ consisting of all right relative BiHom-Hopf modules
that are torsion-free with respect to $\lambda$ and $0$-generated. Then
$\overline F=Q(-\otimes_BA)$ and $G=(-)^{\operatorname{co}H}$ give an
equivalence of categories
\[
\mathcal{M}_B \simeq (\mathcal{M}_A^H)_{\operatorname{tf},0}.
\]
\end{theorem}

\begin{proof}
For every $N\in\mathcal{M}_B$, the object $\overline F(N)$ is
torsion-free by construction. It is $0$-generated because every element
$\overline{n\otimes_Ba}$ can be written as
\[
\overline{n\otimes_Ba}
=
\bigl(\overline{\alpha_N^{-1}(n)\otimes_B1_A}\bigr)\cdot a
\quad
\text{and}
\quad
\overline{\alpha_N^{-1}(n)\otimes_B1_A}
\in
(\overline F(N))^{\operatorname{co}H}.
\]
Thus $\overline F$ takes values in
$(\mathcal{M}_A^H)_{\operatorname{tf},0}$.

By Lemma~\ref{lem:factor-unit-isomorphism}, the unit $N\longrightarrow G\overline F(N)$ is an isomorphism for every $N\in\mathcal{M}_B$. By
Lemma~\ref{lem:factor-counit-isomorphism}, the counit $\overline F G(M)\longrightarrow M$ is an isomorphism for every $M\in(\mathcal{M}_A^H)_{\operatorname{tf},0}$. Therefore, $\overline F$ and $G$ are quasi-inverse equivalences.
\end{proof}

\begin{example}\label{ex:c7-yau-twist}
Let $\Gamma=C_7=\langle u\mid u^7=1\rangle$ and
$\Lambda=C_2=\langle v\mid v^2=1\rangle$; exponents of $u$ and $v$
are read modulo $7$ and $2$, respectively.  Start with the ordinary
group Hopf algebra $\mathbb{k}\Gamma$ and the commuting Hopf
automorphisms
\[
\alpha_H(u^i)=u^{2i},\qquad \beta_H(u^i)=u^{3i}.
\]
Here $\alpha_H\ne\beta_H$ and
$\beta_H\ne\alpha_H^{-1}$.  The standard Yau twist
\cite{GrazianiMakhloufMeniniPanaite2015} has
\begin{align*}
u^i\cdot_Hu^j&=u^{2i+3j},&
\Delta_H(u^i)&=u^{4i}\otimes u^{5i},&
S_H(u^i)&=u^{-i}.
\end{align*}
The functional
\[
\lambda\left(\sum_{i=0}^6c_i u^i\right)=c_0
\]
is a Haar integral. Checking the basis elements $u^i$ verifies both
integral identities, since multiplication by $4$ and by $5$ permutes
$\mathbb Z/7\mathbb Z$.

Let $A=\mathbb{k}[\Gamma\times\Lambda]$ and define
\begin{align*}
\alpha_A(u^i\otimes v^\ell)&=u^{2i}\otimes v^\ell,&
\beta_A(u^i\otimes v^\ell)&=u^{3i}\otimes v^\ell,\\
(u^i\otimes v^\ell)\cdot_A(u^p\otimes v^q)
&=u^{2i+3p}\otimes v^{\ell+q}.
\end{align*}
Define
\[
\rho_A(u^i\otimes v^\ell)
=(u^{4i}\otimes v^\ell)\otimes u^{3i}.
\]
The compatibility identities for $\alpha$ and $\beta$ follow from the
commutativity of the two automorphisms. The coassociativity and counit
identities are verified directly:
\begin{align*}
(\alpha_A^{-1}\otimes\Delta_H)\rho_A(u^i\otimes v^\ell)
&=(u^{2i}\otimes v^\ell)\otimes u^{5i}\otimes u^i\\
&=(\rho_A\otimes\beta_H^{-1})\rho_A(u^i\otimes v^\ell),\\
(\operatorname{id}_A\otimes\varepsilon_H)\rho_A(u^i\otimes v^\ell)
&=u^{4i}\otimes v^\ell
=\alpha_A^{-1}(u^i\otimes v^\ell).
\end{align*}
Moreover,
\begin{align*}
\rho_A\bigl((u^i\otimes v^\ell)\cdot_A(u^p\otimes v^q)\bigr)
&=(u^{i+5p}\otimes v^{\ell+q})\otimes u^{6i+2p}\\
&=\rho_A(u^i\otimes v^\ell)\,\rho_A(u^p\otimes v^q),
\end{align*}
so $A$ is a right $H$-BiHom-comodule algebra.

Its coinvariant algebra and averaging projection are
\begin{align*}
B=A^{\operatorname{co}H}
&=\bigoplus_{\ell=0}^1\mathbb{k}(1\otimes v^\ell)
\cong\mathbb{k}\Lambda,\\
\pi_A(u^i\otimes v^\ell)
&=\begin{cases}
1\otimes v^\ell,&i=0,\\
0,&i\ne0.
\end{cases}
\end{align*}
In particular, $\kappa(A)=0$.  Indeed, if
$0\ne x=\sum c_{i,\ell}u^i\otimes v^\ell$, choose $i_0$ with
$\sum_\ell c_{i_0,\ell}v^\ell\ne0$ and put $a=u^{4i_0}\otimes1$.
Then
$\pi_A(x\cdot_Aa)=\sum_\ell c_{i_0,\ell}(1\otimes v^\ell)\ne0$.
The regular object $A$ is also $0$-generated, since
$a=1_A\cdot_A\beta_A^{-1}(a)$.  Thus $A$ belongs to
$(\mathcal M_A^H)_{\operatorname{tf},0}$, and the equivalence of
Theorem~\ref{thm:main-equivalence} sends the regular $B$-module to an
object isomorphic to $A$.
\end{example}

\begin{corollary}\label{cor:whole-category-equivalence}
If every object of $\mathcal{M}_A^H$ is torsion-free and $0$-generated,
then the original induction functor $F$ and the coinvariant functor $G$ give an equivalence of categories
\[
\mathcal{M}_B\simeq\mathcal{M}_A^H.
\]
\end{corollary}

\begin{proof}
Torsion-freeness gives $\kappa(N\otimes_BA)=0$ for every
$N\in\mathcal M_B$, so $\overline F(N)=F(N)$. The $0$-generation
assumption identifies $(\mathcal M_A^H)_{\operatorname{tf},0}$ with
$\mathcal M_A^H$, and Theorem~\ref{thm:main-equivalence} applies.
\end{proof}

\end{document}